\documentclass[12pt]{amsart}
\usepackage[all]{xy}
\usepackage{amssymb}
\calclayout
\newtheorem{dummy}{anything}[section] 
\newtheorem{theorem}[dummy]{Theorem}
\newtheorem*{thma}{Theorem A}

\newtheorem*{conj}{Conjecture}
\newtheorem*{question}{Question}
\newtheorem{lemma}[dummy]{Lemma} 
\newtheorem{proposition}[dummy]{Proposition} 

\theoremstyle{definition}
\newtheorem{definition}[dummy]{Definition}
 \newtheorem{example}[dummy]{Example}
 
 \newtheorem{remark}[dummy]{Remark}
 \newtheorem*{rem}{Remark}
 \newtheorem*{acknowledgement}{Acknowledgement}

\newcommand{\cA}{\mathcal A}

\newcommand{\bF}{\mathbf F}
\newcommand{\bH}{\mathbf H}
\newcommand{\bZ}{\mathbf Z}
\newcommand{\bQ}{\mathbf Q}

\newcommand{\bbS}{\mathbb S}

\newcommand{\cy}[1]{\bZ/{#1}}

\newcommand{\vv}{\, | \,}
\newcommand{\trf}{tr{\hskip -1.8truept}f}
\newcommand{\ZG}{\bZ G}
\newcommand{\QG}{\bQ G}

\newcommand{\mmatrix}[4]{\left (\vcenter
{\xymatrix@C-2pc@R-2pc{#1&#2\\#3&#4} }
\right )}

\DeclareMathOperator{\Hom}{Hom}
\DeclareMathOperator{\wh}{Wh}
\DeclareMathOperator{\Mod}{mod}
\DeclareMathOperator{\rank}{rank}

\DeclareMathOperator{\Image}{Im}
\DeclareMathOperator{\Ind}{Ind}
\DeclareMathOperator{\Res}{Res}

\newcommand{\nr}{\medskip\noindent $\bullet$\ }
\DeclareMathOperator{\bop}{{\textstyle{\bigoplus}}}
 
  \newcommand{\wwH}{\Sigma}
 \newcommand{\wP}{\Gamma}
 \DeclareMathOperator{\Ext}{Ext}
  \DeclareMathOperator{\wExt}{{\widetilde{Ext}}}
  \DeclareMathOperator{\Aut}{Aut}
 
 \newcommand{\la}{\langle}
  \newcommand{\ra}{\rangle}
   \DeclareMathOperator{\tr}{tr}
   \DeclareMathOperator{\ch}{c}
   \newcommand{\bP}{\mathbf P}
   \newcommand{\CP}{\mathbf C \bP}
   \newcommand{\bnu}{\xi}
   \newcommand{\Bnu}[1]{\xi_{{#1}}}

   \newcommand{\Btau}[1]{\varpi_{{#1}}}
   \newcommand{\MS}[1]{M{#1}}
   
\begin{document}
\title[Free Actions of Finite Groups on $S^n \times S^n$]
{Free Actions of Finite Groups on $S^n \times S^n$}
\author{Ian Hambleton}
\address{Department of Mathematics \& Statistics
\newline\indent
McMaster University
\newline\indent
Hamilton, ON L8S 4K1, Canada}
\email{ian@math.mcmaster.ca}
\author{\"Ozg\"un \"Unl\"u}
\address{Department of Mathematics \& Statistics
\newline\indent
McMaster University
\newline\indent
Hamilton, ON L8S 4K1, Canada}
\email{unluo@math.mcmaster.ca}
\date{Jan.~14, 2009}
\thanks{\hskip -11pt  Research partially supported by NSERC Discovery Grant A4000. }

\begin{abstract} 
Let $p$ be an odd prime. We construct a  non-abelian extension  $\wP $ of $S^1$ by $\cy p \times \cy p$, and prove that any finite subgroup of $\wP $ acts freely and smoothly on $S^{2p-1} \times S^{2p-1}$. In particular, for each odd prime $p$ we obtain  free smooth actions of infinitely many
non-metacyclic rank two $p$-groups on $S^{2p-1} \times S^{2p-1}$. These results arise from a general approach to the existence problem for finite group actions on products of equidimensional spheres.
\end{abstract}
\maketitle

\section*{Introduction}

Conner \cite{conner1} and Heller \cite{heller1} proved that any finite group $G$ acting freely on a product of two spheres must have $\rank G \leqq 2$.  
In other words, the maximal rank of an elementary abelian subgroup of $G$ is at most two. 
However, if both spheres have the same dimension then there are additional restrictions: the alternating group $A_4$ of order $12$ has rank two, but does not admit such an action (see Oliver \cite{oliver1}). It was observed by Adem-Smith \cite[p.~423]{adem-smith}  that $A_4$ is a subgroup of every rank two simple group, so all these are ruled out.

\begin{question}
What group theoretic conditions characterize the rank two finite groups which can act freely and smoothly on $S^n \times S^n$, for some $n\geqq 1$~?
\end{question}

The work of G.~Lewis \cite{lewis2} shows that for every prime $p$, the $p$-Sylow subgroup of a finite group $G$ which acts freely on
$S^n \times S^n$ is abelian unless  $n = 2pr-1$ for some $r\geq 1$.
Lewis also points out \cite[p.~538]{lewis2} that
metacyclic groups  act freely and smoothly on some $S^n \times S^n$, 
but the existence of a free action by any other non-abelian $p$-group, for $p$ odd, has been a long-standing open question. In this paper we provide a general approach to this problem, and construct an infinite family of new examples for each odd prime in the minimal dimension.

\smallskip
For each odd prime $p$, let $\wP$ be the Lie group given by the following presentation
\begin{equation*}
\wP= \left\langle a,b,z\text{ }|\text{ } z \in S^1 \text{, }a^{p}=b^{p}=[a,z]=[b,z]=1\text{, }
[a,b]= \omega\right\rangle
\end{equation*} 
where $\omega = e^{2 \pi i/p } \in S^1 \subseteq \mathbb{C}$. This is a non-abelian central extension of $\cy p \times \cy p$ by $S^1$.
\begin{thma} Let $p$ be an odd prime, and let $G$  be a finite subgroup of $\wP$. 
Then  $G$ acts freely and smoothly on $S^{2p-1} \times S^{2p-1}$. 
\end{thma}
The finite subgroups of $\wP$ which surject onto the quotient $\cy p \times \cy p$ are direct products $G = C \times P(k)$, where $C$ is a finite cyclic group of order prime to $p$, and 
$$
P(k)= \left\langle a,b,c \vv a^{p}=b^{p}=c^{p^{k-2}}=[a,c]=[b,c]=1\text{, }
[a,b]= c^{p^{k-3}} \right\rangle
$$
is a rank two $p$-group of order $p^{k}$, $k \geq 3$. We therefore obtain infinitely many actions of non-metacyclic $p$-groups on $S^{2p-1} \times S^{2p-1}$ for each prime $p$.

\smallskip
An important special case is the extraspecial $p$-group $G_p =P(3)$ of order $p^3$ and exponent $p$. 
Our existence result contradicts claims made in \cite{alzubaidy1}, \cite{alzubaidy2},   \cite{thomas1}, and \cite{yagita1} that $G_p$-actions do not exist (for cohomological reasons) on any product of equidimensional spheres. It was later shown by Benson and Carlson \cite{benson-carlson1} that 
such actions could not be ruled out for any prime $p$ by cohomological methods. Moreover for $p=3$, in \cite{h-unlu1}, we gave an explicit construction of a free smooth action of $\Gamma $ (and in particular $G_3$) on $S^5\times S^5$. This construction provides an alternate proof of Theorem A for $p=3$.

\smallskip
More generally, 
rank two finite $p$-groups were classified by Blackburn \cite{blackburn2} (see also \cite{leary3}).
Consider the additional family, extending the groups $P(k)$:
\begin{equation*}
B(k,\epsilon )= \left\langle a,b,c \vv a^{p}=b^{p}=c^{p^{k-2}}=[b,c]=1\text{, } [a,c]=b \text{, }
[a,b]= c^{\epsilon  p^{k-3}} \right\rangle
\end{equation*} 
where $k\geq 4$, and $\epsilon $ is $1$ or a quadratic non-residue $\Mod p$.
Here is Blackburn's list of the  rank two $p$-groups $G$ with order $p^k$, and $p>3$ (the classification for $p=3$ is more complicated):
\begin{enumerate}
\renewcommand{\labelenumi}{{\bf \Roman{enumi}})}
\item $G$ is a metacyclic $p$-group.
\item $G=P(k)$, for $k\geq 3$.
\item $G = B(k,\epsilon)$, for $k \geq 4$.
\end{enumerate}
 We now know that  groups of types I and II do act freely on a product of equidimensional spheres in the minimal dimension. Is this the complete answer~?
\begin{conj}
Let $p > 3$ be an odd prime. If $G$ is a rank two $p$-group $G$ which acts freely and smoothly on $S^{2pr-1} \times S^{2pr-1}$,  $r \geq 1$, then $G$ 
is metacyclic  or  $G$ is a subgroup of $\wP$.
 \end{conj}

If this conjecture is true, then 
we would know all the possible $p$-Sylow subgroups ($p>3$) for finite groups acting freely on products of equidimensional spheres.
This would be an important step forward in understanding the general problem.

We remark that in order to handle groups of composite order, 
it is necessary to establish the existence of free actions of $p$-groups in higher dimensions $S^{2pr-1} \times S^{2pr-1}$, $r > 1$. 
In \cite{h-unlu1}, we discussed this existence problem specifically for $p=3$, $r=2$, and showed that all odd order subgroups of $SU(3)$, including the extraspecial $3$-group $G_3$ and the type III group $B(4, -1)$, can act freely and smoothly on $S^{11}\times S^{11}$. In particular, we are suggesting that existence results for  $p=3$ will be qualitatively different than those for $p>3$.

\medskip
We can expect an even more complicated structure for the $2$-Sylow subgroup of a finite group acting freely on some $S^n\times S^n$, since this is already the case for free actions on $S^n$. We can take products of periodic groups $G_1\times G_2$ and obtain a variety of actions of non-metacyclic groups  on $S^n\times S^n$ (see \cite{h4} for the existence of these examples, generalizing the results of Stein \cite{stein1}).
Here the $2$-groups are all metabelian, so one might hope that this is the correct restriction on the $2$-Sylow subgroup. 
However, there are non-metabelian $2$-groups which are subgroups of $Sp(2)$, hence by generalizing the notion of fixity in \cite{adem-davis-unlu} to quaternionic fixity, one can construct free actions of these non-metabelian $2$-groups on $S^7\times S^7$ (see \cite{unlu1}).

\begin{rem}
Every rank two finite $p$-group (for $p$ odd)  admits a free
smooth action on some product $S^n \times S^m$, $m \geqq n$ (see \cite{adem-davis-unlu} for $p>3$, \cite{unlu1} for $p=3$). The survey article by A.~Adem \cite{adem1} describes recent progress on the existence problem in this setting for general finite groups (see also \cite{jackson1}). In most cases, construction of the actions requires $m>n$.
\end{rem}

We will always assume that our actions on $S^n\times S^n$ are homologically trivial and $n$ is odd. For free actions of odd $p$-groups this follows from the Lefschetz fixed point theorem.

\begin{acknowledgement} The authors would like to thank Alejandro Adem, Dave Benson, Jim Davis and Matthias Kreck for useful conversations and correspondence.
\end{acknowledgement}

\section{An overview of the proof}
Given a group $G$, and two cohomology classes $\theta_1$, $\theta_2 \in H^{n+1}(G; \bZ)$, we can construct an associated space $B_G$ as the total space of the induced fibration
$$\xymatrix{
K(\bZ\oplus\bZ, n) \ar[r] & B_G \ar[d] &  \cr 
& BG \ar[r]^{\ \ \theta_1 , \theta_2\ \ \ \ \ \ \ \ \ \ \ } & K(\bZ\oplus\bZ, n+1)}$$
where $BG$ denotes the classifying space of $G$. If we also have a stable oriented bundle $\nu_G\colon B_G \to BSO$, then we can consider the bordism groups 
$$\Omega_{k}(B_G,\nu_G)$$
defined as in \cite[Chap.~II]{stong1}. The objects are commutative diagrams
$$\xymatrix{\nu_M \ar[r]^b\ar[d] &\nu_G\ar[d]\cr M \ar[r]^f & B_G}$$
where $M^{k}$ is a closed, smooth $k$-dimensional manifold with stable normal bundle $\nu_M$, $f\colon M \to B_G$ is a reference map, and $b\colon \nu_M \to \nu_G$ is a stable bundle map covering $f$. The bordism relation is the obvious one consistent with the normal data and reference maps. 

Our general strategy is to use the space
 $B_{G}$ as a model for the $n$-type of the orbit space of a possible free $G$-action on  $S^n \times S^n$, and $\nu_G$ as a candidate for its stable normal bundle. For $G$ finite, the actual orbit space (a closed $2n$-dimensional manifold) is obtained by surgery on a representative of a suitable bordism element in $\Omega_{2n}(B_G, \nu_G)$.
 This approach to the problem follows the general outline of Kreck's ``modified surgery" program (see \cite{kreck3}).

 We will carry out this strategy uniformly to prove Theorem A.  We may restrict our attention to the finite subgroups $G\subset\wP$ which surject onto $\cy p \times \cy p$.  Then, by construction, the induced map on classifying spaces gives a circle bundle
$$S^1 \to B_{G} \to B_{\wP}\ .$$
Let $n=2p-1$. We select appropriate data $(\theta_1, \theta_2, \nu_\wP)$ for $\wP$, and then define the data for each $G\subset \wP$ by restriction. We study $B_\wP$ and the bordism groups $\Omega_{2n-1}(B_\wP, \nu_\wP)$ to carry out the following steps.
\begin{enumerate}
\renewcommand{\labelenumi}{(\roman{enumi})}
\item We construct a non-empty subset $T_\wP \subseteq H_{2n-1}(B_\wP;\bZ)$, depending on the data   $( \theta_1, \theta_2, \nu_\wP)$, consisting entirely of primitive elements of infinite order.
\item For each $\gamma \in T_\wP$, we show that there is a bordism element $[N, c] \in \Omega_{2n-1}(B_\wP, \nu_\wP)$ whose image $c_*[N] = \gamma \in T_\wP$ under the Hurewicz map $ \Omega_{2n-1}(B_\wP, \nu_\wP) \to 
H_{2n}(B_\wP;\bZ)$. 
\end{enumerate}
One of the key points is that the cohomology of the groups $\wP$ is much simpler than that of its finite subgroups (see Leary \cite{leary1} for $\wP$, and Lewis \cite{lewis1} for the extra-special $p$-groups), so the computations of Steps (i) and (ii) are best done over $\wP$.

Now for each finite subgroup $G\subset \wP$ as above, define
$T_{G}$ as the image of $T_{\wP}$ under the $S^1$-bundle transfer
$$\trf\colon H_{2n-1}(B_{\wP};\bZ) \to H_{2n}(B_{G};\bZ)$$
induced by the fibration of classifying spaces. 
The subset $T_{G}$ will contain the images of fundamental classes of the possible free $G$-actions on $S^n \times S^n$. For each $\gamma_G = \trf(\gamma) \in T_G$, we have a bordism element $[M, f] \in \Omega_{2n}(B_G, \nu_G)$, where $M$ is the total space of the pulled-back $S^1$-bundle over $[N, c]$  with $c_*[N] = \gamma$. We then show that we can obtain $\widetilde M = S^n \times S^n$ by surgery on $[M,f]$ within its bordism class.

\section{Representations and cohomology of $\wP$}\label{three}

\subsection{Some subgroups of $\wP$} 
For each odd prime $p$, the Lie group $\wP$ is a central extension
$$ 1 \to S^1 \to \wP \to Q_p\to 1$$
where $Q_p = \cy p \times \cy p$. We fix the presentation for $\wP$ given in the Introduction, with generators $\la a, b\ra$ for $Q_p$. For any finite subgroup $G \subset \wP$ which surjects onto $Q_p$, we have a commutative diagram of central extensions
$$\xymatrix{1 \ar[r]& \la c\ra\ar[r]\ar[d]&G\ar[r]\ar[d]& Q_p\ar[r]\ar@{=}[d]& 1\cr
1 \ar[r]& S^1\ar[r]&\wP\ar[r]& Q_p\ar[r]& 1
}$$
where the centre $Z(G) = \la c \ra\subset G$ is a finite cyclic group.
Now we list some subgroups of $\wP$ which will be important in our calculations.
\begin{definition}\mbox{}
Let $d_t =ab^{t}$ if $0 \leq t \leq p-1$, $d_p = b$, and define $D_t = \la d_t\ra$, for $0\leq t \leq p$. Let $\wwH_{t}=\la d_{t},S ^{1}\ra$ denote the subgroup of $\wP$ generated by $d_{t}$ and $S ^{1}$ for $0 \leq t\leq p$.
\end{definition}
We will usually write $\wwH$ instead of $\wwH_{p}$ for the subgroup of $\wP$ generated by $b$ and $S ^{1}$. 
\begin{remark}\label{rem: lifting automorphisms} Since the subgroup $S^1\subset \wP$ is central, any continuous group automorphism
$\phi \in \Aut(\wP)$ induces an automorphism $\bar\phi \in \Aut(Q_p) = GL_2(p)$. In Section \ref{sec: bordism element}, we will use the fact that the image of $\Aut(\wP)$ contains the subgroup  $ SL_2(p)\subset GL_2(p)$.
More explicitly, for each matrix $A \in SL_2(p)$, we can define an automorphism $\phi_A \in \Aut(\wP)$ such that $\bar{\phi}_A= A$ as follows. Any element in $\wP $ can be written
as $a^rb^sz$ for unique $r,s\in \cy p$ and $z\in
S^1$. Given  a matrix 
$A=\left(\vcenter{%
\xymatrix@C-25pt@R-25pt{
  r_{11} & r_{12} \\
  r_{21} & r_{22}}
}\right)$ in $SL_2(p)$ we can define $\phi
_A(a)=a^{r_{11}}b^{r_{12}}$, $\phi _A(b)=a^{r_{21}}b^{r_{22}}$
and $\phi _A(z)=z$. By construction, $\bar{\phi}_A = A$. 
\end{remark}

\subsection{Representations of $\wP$ and some of its subgroups} First, we
define a $1$--dimensional representation 
$$\Phi _{t}\colon\wP\to U(1), \quad \text{for\ }0 \leq t\leq p, $$
 so that 
 $\ker\Phi _{t}=\wwH_{p-t}$ and  $\Phi _{t}(d_{t})=e^{2\pi i/p}$. For any subgroup $G\subset  \wP$ and
$0 \leq t \leq p$ we define a $1$-dimensional
representation of $G$ by the formula:
\begin{equation*}
\Phi _{t,G}=\Res_{G}^{\wP}\left( \Phi _{t} \right) \colon G\to U(1), \quad 0 \leq t \leq p \ .
\end{equation*} 
Second, we define a $1$-dimensional representation $\Phi _{t}'$ of $\wwH_{t}$  by setting:
\begin{equation*}
\Phi _{t}'\colon\wwH_{t}\to U(1), \quad 0 \leq t \leq p,
\end{equation*}
where $ \Phi _{t}'(d_{t})=1$ and $\Phi _{t}'(z)=z$ for $z$ in $S ^{1}$.  For any subgroup 
$ G \subset \wwH_{t}$  we define 
a $1$-dimensional
representation of $G$ by the formula:
\begin{equation*}
\Phi _{t,G}'=\Res_{G}^{\wwH_{t}}\left( \Phi _{t}' \right) \colon G\to U(1),
\quad 0 \leq t \leq p\ .
\end{equation*} 
Finally, we define a $p$--dimensional irreducible representation $\Psi $ of $\wP$ as follows:
\begin{equation*}
\Psi =\Ind_{\wwH}^{\wP}(\Phi'_p)\colon\wP \to SU(p)
\end{equation*}
and for  any subgroup $G \subset
\wP$ we define:
 \begin{equation*}
\Psi _{G}=\Res_{G}^{\wP}\left( \Psi \right) \colon G\to
SU(p)
\end{equation*}
by restriction, as a $p$-dimensional
representation of $G$.

\subsection{Cohomology of $\wP$ and some of its subgroups} 
We will use the notations and results of Leary \cite{ leary1} for the integral
cohomology ring of $\wP$. 
\begin{theorem}[{\cite[Theorem 2]{ leary1}}]
\label{TheoremofLeary} $H^{\ast }(B\wP;\bZ)$ is generated by
elements $\alpha ,\beta,\sigma _{1},\chi _{2}$, $\dots$, $\chi _{p-1}$, $\zeta $, with
\begin{equation*}
\deg(\alpha )=\deg(\beta )=2\text{, \ }\deg(\zeta )=2p\text{, \ \  }
\deg(\chi _{i})=2i,
\end{equation*}
subject to some relations.
\end{theorem}
\noindent In the statement of Theorem \ref{TheoremofLeary}, the elements $\alpha =\Phi _{0}\colon
\wP\to U(1)$ and $\beta =\Phi _{p}\colon\wP \to U(1)$, by considering $H^{2}(B\wP;\bZ  
)=\Hom(\wP,S ^{1})$,  and $\zeta $ is the $p^{th}$ Chern
class of the $p$--dimensional irreducible representation $\Psi $ of $
\wP$. The mod $p$ cohomology ring of $\wP$ is also
given by Leary:
\begin{theorem}[{\cite[Theorem 2]{leary2}}]
\label{TheoremofLearymodp} $H^{\ast }(B\wP;\bZ /p)$ is generated by elements 
$y$, $y^{\prime }$, $x$, $x^{\prime}$, $c_{2},c_{3},\dots ,c_{p-1},z$, with
\begin{eqnarray*}
\deg(y) &=&\deg(y^{\prime })=1\text{, }\deg(x)=\deg(x^{\prime })=2\text{,} \\
\text{ }\deg(z) &=&2p\text{, \ \ and \ \ \  }\deg(c_{i})=2i,
\end{eqnarray*} 
subject to some relations.
\end{theorem}
\noindent Let $\pi _{\ast }$ stand for the projection map from 
$H^{\ast }(B\wP; \bZ )$ to $H^{\ast }(B\wP; \bZ /p)$, and $\delta _{p}$ for the Bockstein from $H^{\ast
}(B\wP; \bZ /p)$ to $H^{\ast +1}(B\wP; \bZ )$ then $\delta_{p}(y)=\alpha$, $\delta_{p}(y^{\prime })=\beta $, $\pi _{\ast }(\alpha )=x$, $\pi _{\ast }(\beta )=x^{\prime }$, $\pi _{\ast }(\chi _{i})=c_{i}$, and $\pi _{\ast }(\zeta )=z$. Here are some facts about the cohomology of certain subgroups.
\begin{remark} \label{CohomologyOfSubgroups}Considering $H^{2}(BG,\bZ )=\Hom(G,S ^{1})$
\begin{enumerate}
\item $H^{*}(BS^1;\bZ )=\bZ [\tau ]$ where $\tau =\Phi _{t,S^1}' $. So $\tau $ is the identity map on $S^1$.
\item $H^{*}(B\wwH_{t};\bZ )=\bZ[\tau ',v '\vv pv'=0 ]$ where $\tau '=\Phi _{t}'$ and $v '=\Phi _{t,\wwH_{t}}$
\item $H^{*}(B\wwH_{t}, \bZ /p)=\bF_p[\bar{\tau }] \otimes ( \Lambda (u) \otimes \bF_p[v] )$ where $\bar{\tau }$ and $v$ are mod $p$ reductions of 
$\tau '$ and $v'$ respectively and $\beta(u)=v$.
\end{enumerate}
\end{remark}

We calculate some restriction maps:
$$\Res^{\wP }_{\wwH _t}(\alpha )=
\begin{cases} 
v'           & \textup{if\ \ }  0 \leq t \leq p-1, \cr
0           & \textup{if\ \ } t=p\ .
\end{cases}
\text{\ \ \ and \ \ \ }
\Res^{\wP }_{\wwH _t}(\beta )=
\begin{cases} 
tv'           & \textup{if\ \ }  0 \leq t \leq p-1, \cr
v'           & \textup{if\ \ } t=p\ .
\end{cases}
$$
The property
$$\Res^{\wP }_{\wwH _t}(\alpha ^{p}-\alpha ^{p-1}\beta +\beta ^{p})=(v')^{p}, $$
for $0 \leq t \leq p$, shows that this element is a good candidate for a $k$-invariant.

\section{The $(2p-1)$-type $B_{\wP}$ and the bundle data}

We now construct the space $B_{\wP}$ needed as a model for the $(2p-1)$-type of the quotient space of our action. Then we construct a 
bundle $\nu_{\wP}$ over this space $B_{\wP}$ which will pullback to the normal bundle of the quotient space of this action.

\subsection{Definition of $B_{\wP}$ } \label{TheDefitionOfB_G}
We fix the element
\begin{equation*}
k =\theta_{1} \oplus \theta_{2} =\zeta \oplus \left( \alpha ^{p}-\alpha ^{p-1}\beta +\beta ^{p}\right) \in H^{2p}(\wP; 
\bZ  )\oplus H^{2p}(\wP;\bZ  )\ .
\end{equation*} 
  For  any
subgroup  $G \subset\wP$ define
\begin{equation*}
k_{G}=\Res_{G}^{\wP}(k ) \in  
H^{2p}(G;\bZ\oplus \bZ),
\end{equation*} 
and define $\pi_{G}$ as the fibration classified by $k _{G}$:
$$\xymatrix@R-5pt@C-7pt{K(\bZ\oplus\bZ, 2p-1) \ar[r] & B_G \ar[d]^{\pi_G} &{\hphantom{xxxxxxxxxxxxx}}\cr &BG\ar[r]^(0.3){k_G} & 
 K(\bZ\oplus\bZ, 2p)\ .
}$$
Note that the natural map $BG \to B\wP$, induced by the inclusion, gives a diagram
$$\xymatrix{
B_G\ar[r]\ar[d]^{\pi_G}& B_{\wP}\ar[d]^{\pi_{\wP}}\cr
BG \ar[r]& B\wP
}$$
which is a pull-back square.

\subsection{The bundle data  over  $B_{\wP}$} 
For any subgroup $G\subseteq \wP$ we will define two bundles $\Btau{G}$ and $\Bnu{G}$ over $BG$, which will pull back by the classifying map to the stable tangent and normal bundle  respectively of the quotient of a possible $G$-action on $S^n \times S^n$. The pullbacks of these bundles over $BG$ to bundles over $B_{G}$ will be denoted by $\tau_G$ and $\nu_G$ respectively.

\medskip
\noindent
\textbf{(1) Tangent bundles:}
We have the representations $\Psi_G\colon G \to SU(p)$ and $\Phi_{t, G}\colon G \to U(1)$. Let  $\psi _{G}$ denote the $p$--dimensional complex vector bundle classified by 
$$\psi _{G}=B\Psi _{G}\colon BG \to BSU(p),$$
and let $\phi _{t,G}$ denote the complex line bundle classified by 
$$\phi _{t,G}=B\Phi _{t,G}\colon BG\to BU(1) = BS^1\ . $$
We define a $3p$--dimensional complex vector bundle $\Btau{G}$ on $BG$ by the  Whitney sum
\begin{equation*}
\Btau{G}= \psi _{G} \oplus \phi _{0,G}^{\oplus p} \oplus \phi _{p,G}^{\oplus p}
\end{equation*}
and use the same notation for the stable vector bundle $\Btau{G}\colon BG \to BSO$. 
We now identify our candidate $\tau_G$ for the stable tangent bundle.
\begin{definition}\label{bundlechoices_tangent}
Let $\tau _{G}$ denote the stable vector bundle  on $B_{G}$ classified by the
composition
 $$\tau_G\colon B_G \xrightarrow{\pi _{G}} BG \xrightarrow{\Btau{G}} BSO\ .$$
\end{definition}

\medskip
\noindent
\textbf{(2) Normal bundles:}
First we show that there is an stable inverse of the vector bundle $\Btau{G}$ over $BG$, when restricted to a finite skeleton of $BG$.
\begin{lemma}
For any subgroup $G \subseteq \wP$, there exists a stable bundle $\Bnu{G}\colon BG \to BSO$, such that $\Bnu{G} \oplus \Btau{G}= \varepsilon$, the trivial bundle, when restricted to the $(4p-1)$-skeleton of $BG$.
\end{lemma}
\begin{proof}
Take $N=4p-1$ and let $\Btau{\wP} | _{B\wP^{(N)}} $ denote the pull-back of $ \Btau{\wP} $ to $B\wP^{(N)}$, the $N$--th skeleton of $B\wP$, by the inclusion 
map of $B\wP^{(N)}$ in $B\wP$. Then there exists a vector bundle $\Bnu{\wP} $ over $B _{\wP}^{(N)}$ such that the bundle $\Bnu{\wP}  \oplus (\Btau{\wP} | _{B\wP^{(N)}} ) $ is trivial over  
$B\wP^{(N)}$, since $B\wP^{(N)}$ is a finite CW--complex. Stably this vector bundle   is classified by a map $\Bnu{\wP} \colon B\wP^{(N)} \to BU$ and there is no 
obstruction to extending this classifying map to a map $B\wP \to BU$, as the obstructions to doing so lie in the cohomology groups 
$$ H^{*+1}(B\wP,B\wP^{(N)};\pi _{*}(BU))=0 \ .$$ 
We will use the same notation $\Bnu{\wP}$ to denote the stable vector bundle classified by any extension map  $B\wP\to BU\to BSO$.
We  then define $$\Bnu{G}\colon BG \to BSO$$ by composition with the map $BG \to B\wP$ induced by $G\subseteq \wP$.
\end{proof}

We now identify our candidate $\nu_G$ for the stable normal bundle.
\begin{definition}\label{bundlechoices}
Let $\nu_G$ denote the stable vector bundle  on $B_{G}$ classified by the
composition
$$ \nu_G\colon B_G \xrightarrow{\pi _{G}}  BG \xrightarrow{\Bnu{G}}  BSO\ .$$
\end{definition}

\subsection{Characteristic classes}
We will now calculate some  characteristic classes for the bundles $\Btau{\wwH_t}$ and $\Bnu{\wwH_t}$ over $B\wwH_t$. The total Chern class of a bundle $\xi$ will be denoted $c(\xi)$. See \cite[p.~228]{milnor-stasheff} for the definition of the mod $p$ Wu classes 
$q_{k}(\bnu )\in H^{2(p-1)k}(B;\bZ/p)$. 
\begin{lemma}\label{chernclassoverBH} The total Chern class of $\Btau{\wwH_t}$ is
$$\ch (\Btau{\wwH_t}) = \ch (\psi _{\wwH_t}) \ch (\phi _{0,\wwH_t}^{\oplus p}\oplus \phi _{p,\wwH_t}^{\oplus p} )$$
where
\begin{enumerate}
\item $\ch (\psi _{\wwH_t}) = 1 - (v')^{p-1} + ((\tau')^p - (v')^{p-1}\tau')$
\item $\ch (\phi _{0,\wwH_t}^{\oplus p}\oplus \phi _{p,\wwH_t}^{\oplus p} ) =  1 + (1+t)(v')^p + t(v')^{2p}$
\end{enumerate}
\end{lemma}
\begin{proof} 
Given two $1$-dimensional representation $\varPhi\colon G \to S ^1$ and $\varPhi ' \colon G \to S ^1$ and a natural number $k$, we will write $\varPhi ^k(g)=(\varPhi (g))^k$ and $(\varPhi  \varPhi ') (g)=\varPhi (g)\varPhi'(g)$.
 It is easy to see that 
$$\Psi _{\wwH_t}=\Phi _{t}'\oplus \Phi _{t,\wwH_t}\Phi _{t}'\oplus \Phi _{t,\wwH_t}^{2}\Phi _{t}' \oplus \dots \oplus \Phi _{t,\wwH_t}^{p-1}\Phi _{t}'\ .$$
Hence the total Chern class of $\psi _{\wwH_t}$ is
$$(1+\tau')(1+v'+\tau')(1+2v'+\tau')\dots (1+(p-1)v'+\tau')=1 - (v')^{p-1} + (\tau')^p - (v')^{p-1}\tau'$$
since $pv'=0$.

We have  $\ch (\phi _{0,\wwH_t}^{\oplus p})=(1+v')^{p}$ when $0 \leq t \leq p-1$ ($1$ when $t=p$), and  $\ch (\phi _{p,\wwH_t}^{\oplus p})=(1+tv')^{p} $ when $0 \leq t \leq p-1$ (but $(1+v')^p$ when $t=p$). Hence the total Chern class of $\phi _{0,\wwH_t}^{\oplus p}\oplus \phi _{p,\wwH_t}^{\oplus p}$ is equal to
$$(1+v')^{p}(1+tv')^{p}=(1+(v')^{p})(1+(tv')^p)=(1 + (1+t)(v')^p + t(v')^{2p})$$
when $0 \leq t \leq p-1$ and it is equal to 
$$(1+v')^p=(1+(v')^{p})=(1 + (1+t)(v')^p + t(v')^{2p})$$
when $t=p$.
\end{proof}

Now we will calculate the total Chern class of the bundle over $B\wwH_t$ that pulls backs to the normal bundle.

\begin{lemma}\label{chernclassesofbnu}
The total Chern class of $\Bnu{\wwH_t}$ is
$$\ch (\Bnu{\wwH_t}) = 1 + (v')^{p-1} + \text{\ higher terms \ }  $$
\end{lemma}
\begin{proof}
By Lemma \ref{chernclassoverBH} we know that the total Chern class of $\Btau{\wwH_t}$ is
$$\ch ( \Btau{\wwH_t}) = 1 - (v')^{p-1} +\text{\ higher terms \ } $$ 
By the construction of $\Bnu{\wwH_t}$, we know that $\Bnu{\wwH_t}\oplus \Btau{\wwH_t}$ is a trivial bundle over $B\wwH_t^{(4p-1)}$, and the result follows.
\end{proof}

For the rest of this section set $r=\frac{p-1}{2}$.
\begin{lemma} \label{pontrjaginclassesofbnu} The first few Pontrjagin classes of the bundle $\Bnu{\wwH_t}$ are as follows 
$$p_{k}(\Bnu{ \wwH_t})=
\begin{cases} 
1           & \textup{if\ \ } k=0,\cr
0           & \textup{if\ \ } 0< k < r,\cr
(-1)^{r}2(v')^{p-1}  & \textup{if\ \ } k = r\ .
\end{cases}$$
\end{lemma}
\begin{proof}
This is direct calculation given Lemma \ref{chernclassesofbnu} and the fact that
$$p_{k}(\Bnu{\wwH_t})= \ch _{k}(\Bnu{\wwH_t})^2-2\ch _{k-1}(\Bnu{\wwH_t})\ch _{k+1}(\Bnu{\wwH_t})+-\dots \mp 2\ch _{1}(\Bnu{\wwH_t})\ch _{2k-1}(\Bnu{\wwH_t}) \pm 2 \ch _{2k}(\Bnu{\wwH_t})\ .\qedhere$$
\end{proof}

The main result of this section is the following:

\begin{lemma}\label{sphericalclass}
$q_{1}(\Bnu{\wwH_t})=v^{p-1}\in H^{2(p-1)}(B\wwH_t;\bZ/p)$ 
\end{lemma}
\begin{proof}
Let $\{K_{n}\}$ be the multiplicative sequence belonging to the polynomial $f(t)=1+t^r$. A result of Wu shows (see Theorem 19.7 in \cite{milnor-stasheff}) that $$q_{1}(\Bnu{\wwH_t})=K_{r}(p_{1}(\Bnu{\wwH_t}),\dots,p_{r}(\Bnu{\wwH_t}))\Mod p \ .$$ By Lemma \ref{pontrjaginclassesofbnu} we know that $p_{1}(\Bnu{\wwH_t}),\dots,p_{r-1}(\Bnu{\wwH_t})$ are all zero, hence we are only interested in the coefficient of $x_r$ in the polynomial $K_r(x_1,\dots,x_r)$. By Problem 19-B in \cite{milnor-stasheff} this coefficient is equal to $s_r(0,0,\dots,0,1)=(-1)^{r+1}r$ (see \cite[p.~188]{milnor-stasheff})
Hence we have
$$q_{1}(\Bnu{\wwH_t})=(-1)^{r+1}r\bar{p}_{r}(\Bnu{\wwH_t})=(-1)^{r+1}r(-1)^{r}2v^{p-1}=(-1)(p-1)v^{p-1}=v^{p-1}$$
where $\bar{p}_{r}(\Bnu{\wwH_t})$ denotes the mod $p$ reduction of $p_{r}(\Bnu{\wwH_t})$.
\end{proof}

\section{Smooth models $M_t$ and $N_t$}

Here we construct free smooth actions of the subgroups $\wwH_t$ and $D_t$ on $S^{2p-1} \times S^{2p-1}$ and $S^{4p-3}$ respectively,  with the right bundle data. These provide models for covering spaces of the actions we are trying to construct.

\subsection{Construction of the examples $M_t$ and $N_t$}\label{mainexamples}

Given an $m$--dimensional representation $\varPhi \colon G\to U(m)$ of a
group $G$, we have an induced $G$--action on $\mathbb{C}^{m}$, and the space $ 
S(\varPhi ) = S^{2m-1}$ will be the $G$--equivariant unit sphere in $\mathbb{C}^{m}$. We now construct two main examples.
\begin{enumerate}
\item
For $G = \wwH_t$ and $t\in \{0,\dots,p\}$, define 
\begin{equation*}
M_{t}=(S(\Psi _{\wwH_{t}})\times S((\Phi _{t,\wwH_{t}})^{\oplus p})/\wwH_t = (S^{2p-1} \times S^{2p-1})/\wwH_t
\end{equation*}

\item For $G=D_t$ and $t\in \{0,\dots,p\}$, define 
\begin{equation*}
N_t=S(\Phi _{t,D_{t}}\oplus  \Phi _{t,D_{t}}^2\oplus\dots\oplus \Phi _{t,D_{t}}^{p-1}\oplus (\Phi _{t,D_{t}})^{\oplus p})/D_t =  S^{4p-3}/D_t
\end{equation*}
where, for a $1$-dimensional representation $\varPhi $, we set $\varPhi ^{k}$ to be the $k^{th}$ power of $\varPhi $ induced by the multiplication in $S ^{1}$. In other words, if $\varPhi (v)=\lambda v$ then we set $\varPhi ^{k}(v)=\lambda^kv$. 
\end{enumerate}

\subsection{The $(2p-1)$-type of $M_{t}$}
Let $X^{[n]}$ denote the $n$-type of a space $X$.
\begin{lemma}\label{2p-2TypeOfMt}
$M_{t}^{[2p-2]}\simeq B\wwH_{t}$ and the composition $M_{t}\xrightarrow{i_{2p-2}} M_{t}^{[2p-2]}\xrightarrow{\simeq }B\wwH_{t}$ is homotopy equivalent to the classifying map $c_t\colon M_t \to B\wwH_t$.  
\end{lemma}
\begin{proof}
This follows as 
$S(\Psi _{\wwH}, (\Phi _{t,\wwH_{t}})^{\oplus p})=S^{2p-1}\times S^{2p-1}$ is 
$(2p-2)$--connected and the action of $\wwH_{t}$ on $S(\Psi _{ 
\wwH_{t}}, (\Phi _{t,\wwH_{t}})^{\oplus p})$ is free.
\end{proof}
\begin{lemma}\label{2p-1TypeOfMt}
$M_{t}^{[2p-1]}\simeq B_{\wwH_{t}}$.
\end{lemma}
\begin{proof}
Let $\ch _{p}(\varPhi )$ denote the $p^{th}$ Chern class of a representation $\varPhi $ then 
$$k_{2p-1}(M_{t})  = \ch _{p}( \Psi _{\wwH_{t}}) \oplus \ch _{p} ( ( \Phi _{t,\wwH_{t} } )^{\oplus p} )= \Res_{\wwH_{t}}^{\wP}(\zeta ) \oplus \Res_{\wwH_{t}}^{ \wP}(\alpha ^{p}-\alpha ^{p-1}\beta +\beta ^{p})=k_{\wwH_{t}}$$ 
Hence the results follows by Lemma \ref{2p-2TypeOfMt}.
\end{proof}
\subsection{The tangent bundle of $M_{t}$} 
\begin{lemma} The tangent bundle of $M_t$ is stably equivalent to the pull-back of $\tau _{\wwH_{t}}\colon B_{\wwH_t} \to BSO$ (see \textup{Definition \ref{bundlechoices_tangent}}).
\end{lemma}
\begin{proof}
The tangent
bundle $T(M_{t})$ of $M_{t}$ clearly fits into the following pull-back diagram 
$$\xymatrix{T(M_{t})\oplus \varepsilon \ar[r]\ar[d] &  E\wwH_{t}\times _{ 
\wwH_{t}}(\mathbb{C}^{p}\times \mathbb{C}^{p})\ar[d]^\pi \\ 
M_{t}\ar[r]^{c_t} & B\wwH_{t}
}$$
where $\varepsilon $ is a trivial bundle over $M_{t}$ and the action of $\wwH_{t}$ on $\mathbb{C}^{p}\times \mathbb{C}^{p}$ is given by 
$\Psi _{\wwH_{t}}$ and $(\Phi _{t,\wwH_{t}})^{\oplus p}$ respectively. Hence we have 
$c_t^{*}([\pi ])=c_t^{*}([\psi _{\wwH_{t}}]+p[\phi _{t,\wwH_{t}}])$ in complex $K$-theory $\widetilde{K}(M_{t})$. However, $M_{t}$ fits into the 
following pull-back diagram:
$$\xymatrix@R-2pt{M_{t} \ar[r]\ar[d] & ED_{t}\times _{D_{t}}\CP^{p-1} \ar[d]\\ 
L_{t}\ar[r] & BD_{t}
}$$
where $L_{t}=S(\Phi _{t,D_{t}}^{\oplus p})/D_{t}$ and the action of $D_{t}$ on $\CP^{p-1}$ is induced by action of $D_{t}$ on $\mathbb{C}^{p}$ given by $\Psi _{D_{t}}$.
Hence $\widetilde{K}(M_{t})$ is a $\widetilde{K}(L_{t})$--module by Proposition 2.13 in Chapter IV in \cite{karoubi1} and the exponent of $\widetilde{K}(L_{t})$   is $p$ (see Theorem 2 in \cite{kambe1}). Hence the exponent of $\widetilde{K}(M_{t})$ is $p$ and
$c_t^{*}([\pi ])=c_t^{*}([\psi _{\wwH_{t}}])=c_t^{*}([\Btau{\wwH_{t}}]) $. This means the tangent bundle of $M_{t}$ is 
stably equivalent to the pull-back of the bundle $\Btau{\wwH_{t}}$ over $B\wwH_{t}$ by the classifying map. However, 
by Lemma \ref{2p-2TypeOfMt} and Lemma \ref{2p-1TypeOfMt}, we know that the classifying map is homotopy equivalent to the compositon 
$\xymatrix{M_{t}\  \ar[r]^(0.4){i_{2p-1}}& M_{t}^{[2p-1]} \ar[r]^(0.6){\simeq }&  B_{\wwH_{t}} \ar[r]^{\pi_{\wwH_{t}}} & B\wwH_{t}}$.
Hence, the tangent bundle of $M_{t}$ is stably equivalent to the pull-back of $\tau _{\wwH_{t}}$.
\end{proof}

\subsection{The tangent bundle of $N_{t}$} 
\begin{lemma}
The tangent bundle of $N_t$ is stably equivalent to the pull-back of 
$\Btau{D_{t}}\colon BD_t \to BSO$.
\end{lemma}
\begin{proof}
The tangent
bundle $T(N_{t})$ of $N_{t}$ clearly fits into the following pull-back diagram 
$$\xymatrix@R-2pt{T(N_{t})\oplus \varepsilon \ar[r]\ar[d] & ED_{t}\times _{D_{t}}\mathbb{C}^{2p-1} \ar[d]\\ 
N_{t}\ar[r] &  BD_{t}
}$$
where $\varepsilon $ is a trivial bundle over $N_{t}$ and the action of $D_{t}$ on $\mathbb{C}^{2p-1}$ is given by 
$ \Phi _{t,D_{t}} \oplus \Phi _{t,D_{t}}^2 \oplus \dots  \oplus \Phi _{t,D_{t}}^{p-1} \oplus  (\Phi _{t,D_{t}})^{\oplus p}$, 
where  $\varPhi ^k(g)=(\varPhi (g))^k$ for a 1-dimensional representation $\varPhi \colon G \to S ^1$. Now it is easy to see that  
$$\Btau{D_{t}} =1 \oplus \Phi _{t,D_{t}} \oplus \Phi _{t,D_{t}}^2 \oplus \dots  \oplus \Phi _{t,D_{t}}^{p-1}\oplus  (\Phi _{t,D_{t}})^{\oplus p}\ .$$
Hence it is clear that the tangent bundle of $N_{t}$ is the stably equivalent to the 
pull-back of $\Btau{D_{t}}$.
\end{proof}

\section{The image of the fundamental class}

In this section we define a subset $T_{\wP} \subset H_{4p-3}(B_{\wP};\bZ)$, which will turn out to contain the images of all the possible fundamental classes for our actions. To show that this set is non-empty, and consists of primitive elements of infinite order, we need to compute some cohomology groups of $B_{G}$, $G \subset \wP$. To carry out these computations, we will use the cohomology Serre spectral sequence of the fibration $ K \longrightarrow B _{G} \xrightarrow{\pi _{G}}  BG $,
where $K=K(\bZ\oplus \bZ,2p-1)$.

\subsection{Definition of $\gamma _{S^1}$, $\gamma _{\wwH _t}$ and $T_{\wP}$} \label{SomeClasses}
First note that  the universal cover of $M_{t}$ is $ \CP^{p-1} \times S^{2p-1} $, for $0\leq t\leq p$, and the universal cover of $B_{\wwH_{t}}$ is $B_{S^1}$. Hence, we can assume that we have the following pull-back diagram where the map $c$ does not depend on $t$. 
\[\xymatrix{\CP^{p-1} \times S^{2p-1} \ar[r]^(0.6)c\ar[d]& B_{S^1}\ar[d]\\ 
M_{t} \ar[r]^{c_{t}} & B_{\wwH_{t}}
}\]
We define an element $\gamma _{S^{1}}$ in $H_{4p-3}(B_{S^{1}};\bZ)$ as the image of the fundamental class of 
$\CP^{p-1} \times S^{2p-1}$ under the map $c$ defined in the above diagram.  In other words
 $$ \gamma _{S^{1}}=c_{*}\big [\CP^{p-1} \times S^{2p-1}\big ]\in H_{4p-3}(B_{S^{1}};\bZ)\ .$$ 
Similarly, we define $\gamma _{\wwH_{t}}\in H_{4p-3}(B_{\wwH_{t}};\bZŒ)$, $0\leq t\leq p$,  as the image of the fundamental class of $M_t$. In other words
$$ \gamma _{\wwH_{t}} =(c_{t})_{*}\big [M_{t} \big ]\in  H_{4p-3}(B_{\wwH_{t}};\bZ)$$
\begin{definition}\label{def: define TsubGamma}
We define 
$$ T_{\wP}=\{ \gamma \in H_{4p-3}(B_{\wP}; \bZ ) \vv\  p\cdot(\tr(\gamma )-\gamma _{S^1})=0  \} $$ 
where 
$\tr\colon H_{4p-3}(B_{\wP}; \bZ )\to H_{4p-3}(B_{S^{1}}; \bZ )$ 
denotes the transfer map. 
\end{definition}
One of our main tasks will be to show that this subset
$T_{\wP}$ is non-empty ! 

\subsection{The (co)homology of $K$}

Let $A$ be an abelian group. We will write $$ _{(p)}A := A/\la \text{torsion prime to } p \ra  \ .$$
The map $$P^1\colon H^i (K;\cy p) \to H^{i + 2(p-1)}(K; \cy p)$$ is the first mod $p$ Steenrod operation and the map $$\delta \colon H^{i}(K;\bZ /p)\to H^{i+1}(K;\bZ )$$ is the Bockstein homomorphism.
The following lemma gives the cohomology groups of $K$ in the range we need.

\begin{lemma}\label{CohomologyOfK}
Denote $H^{2p-1}(K;\bZ)=\la z_{1},z_{2} \ra =\bZ \oplus \bZ $ and let $\bar{z}_{1}$ and $\bar{z}_{2}$ be the mod $p$ reductions of $z_1$ and $z_2$ respectively. We have
\begin{enumerate}
\item $H^{i}(K;A)$ is a torsion group for $2p \leq i \leq 4p-3$.
\item $_{(p)}H^{i}(K;A)=0$ for $2p \leq i \leq 4p-4$.
\item $_{(p)}H^{4p-3}(K;\bZ )=0$.
\item $_{(p)}H^{4p-2}(K;\bZ)=\la z_1 \cup z_2, \delta (P^1(\bar{z}_{1})), \delta (P^1(\bar{z}_{2})) \ra =\bZ \oplus \bZ /p \oplus \bZ /p \ .$
\item $H^{4p-1}(K;\bZ)$ has no $p$-torsion.
\end{enumerate}
\end{lemma}

\begin{proof}
See results of Cartan \cite{cartan2}, \cite{cartan3}.
\end{proof}

We will also need some information about the homology of  $K$.

\begin{lemma} \label{HomologyOfK} We have
\begin{enumerate}
\item $H_{2p-1}(K;\bZ)=\bZ \oplus \bZ $,
\item $H_{2p-1}(K,;\bZ/p)=\bZ/p \oplus \bZ/p $, 
\item $_{(p)}H_{i}(K;A)$ is $0$ for $2p-1<i<4p-3$, and
\item $_{(p)}H_{4p-3}(K;\bZ)=\bZ /p \oplus \bZ /p $.
\end{enumerate}
\end{lemma}
\begin{proof}
See results of Cartan \cite{cartan2}, \cite{cartan3}.
\end{proof}

\subsection{The (co)homology spectral sequences} \label{SectionOnHomologySerreSpectralSequence}
For any subgroup $G\subset \wP$, let $R$ be a ring and  
$$ \{E_{n,m}^{r}(G,R),d^{r} \}\text{\ \ \ and \ \ \ }\{E^{n,m}_{r}(G,R),d_{r}\}$$ be the homology and the cohomology  Serre spectral sequences (respectively) of the fibration $$ K \longrightarrow B _{G} \xrightarrow{\pi_{G}}  BG $$
where $K=K(\bZ\oplus \bZ,2p-1)$. The second page of these spectral sequence is given by: 
\begin{equation*}
E_{n,m}^{2}(G,R)=H_{n}(BG;H_{m}(K;R ))\text{\ \ \ and \ \ \ }E^{n,m}_{2}(G,R)=H^{n}(BG;H^{m}(K;R)), 
\end{equation*} 
and they converge to $H_{*}(B_{G};R ) $ and to $H^{*}(B_{G};R)$ respectively.
The filtration $F_{*}H_{*}(B_{G};R )$ of $H_{*}(B_{G};R )$ is given by
\begin{equation*}
F_{n}H_{n+m}(B_{G};R )=\Image\left\{ H_{n+m}(B_{G}^{\{n\}};R )\xrightarrow{(i_{n})_{*}}
H_{n+m}(B_{G};R )\right\}
\end{equation*}
and the filtration $F^{*}H^{*}(B_{G};R )$ of $H^{*}(B_{G};R )$ is given by
\begin{equation*}
F^{n}H^{n+m}(B_{G};R)=\ker\left\{H^{n+m}(B_{G};R)\xrightarrow{(i_{n})^{*}}  H^{n+m}(B_{G}^{\{n-1\}};R)\right\}
\end{equation*}
respectively, where 
 $$B_{G}^{\{n\}}=\pi_{G}^{-1}(BG^{(n)})\ .$$
When $R=\bZ$, we will write $ \{E_{n,m}^{r}(G),d_{r}\}$ and $ \{E^{n,m}_{r}(G),d_{r}\}$ instead of $ \{E_{n,m}^{r}(G,\bZ),d_{r}\} $ and $ \{E^{n,m}_{r}(G,\bZ),d_{r}\} $ respectively. The cohomology groups for $$ E^{*,0}_{2}(G,R)=H^{*}(BG;R) $$ are given in Theorem \ref{TheoremofLeary}, Theorem \ref{TheoremofLearymodp}, and Remark \ref{CohomologyOfSubgroups},  and the calculation of  $$E^{0,*}_{2}(G,R)=H^{*}(K;R)$$ is given in Lemma \ref{CohomologyOfK}.

\subsection{Definition of $z_G$, $z'_G$, and $Z _G$}  \label{SectionOnCohB_G}

Let $G$ be a subgroup of $\wP$ which contains $S^1$. We have 
$$E_{2p}^{0,2p-1}(G)=H^{2p-1}(K;\bZ )=\la z_1, z_2 \ra$$ and we can assume that 
$$ d_{2p}(z_1) = \begin{cases} 
\zeta           & \textup{if\ \ } G=\wP,\cr
(\tau ')^{p} - (v ' )^{p-1}\tau '            & \textup{if\ \ } G=\wwH_t,\cr
\tau ^{p}  & \textup{if\ \ } G=S^1
\end{cases}
\text{ \ and \ } 
d_{2p}(z_2) =\begin{cases} 
\alpha ^{p}-\alpha ^{p-1}\beta + \beta ^{p}           & \textup{if\ \ } G=\wP,\cr
(v ' )^{p}            & \textup{if\ \ } G=\wwH_t,\cr
0  & \textup{if\ \ } G=S^1
\end{cases}$$
where $d_{2p}(z_1)$ and $d_{2p}(z_2)$ are in $E_{2p}^{2p,0}(G)=H^{2p}(BG; \bZ )$.
Hence there exists $z_G$ in $H^{2p-1}(B_{G};\bZ )$ such that
$$H^{2p-1}(B_{G};\bZ )=\la z_G\ra = \bZ \text{\ \ \ and \ \ \ } i^{*}(z_{G})=
\begin{cases} 
 pz_{2}         & \textup{if\ \ } G=\wP,\cr
 pz_{2}         & \textup{if\ \ } G=\wwH_t,\cr
 z_{2}          & \textup{if\ \ } G=S^1
\end{cases} $$ 
where $i\colon K \to B_G$ is the inclusion map. Moreover we have 
$$H^{2p-2}(B_G;\bZ )=H^{2p-2}(B_G;\bZ )=
\begin{cases} 
  \la  \alpha ^{p-1}, \alpha ^{p-2}\beta , \dots,  \beta^{p-1}, \chi_{p-1}  \ra         & \textup{if\ \ } G=\wP,\cr
 \la  (v ' ) ^{p-1}, (v ' ) ^{p-2}\tau '  , \dots,  (\tau ' )^{p-1} \ra           & \textup{if\ \ } G=\wwH_t,\cr
\la \tau ^{p-1} \ra   & \textup{if\ \ } G=S^1
\end{cases} $$
Define $z' _G$ in $H^{2p-2}(B_G,\bZ )$ as follows
$$z'_G =
\begin{cases} 
(\pi_{\wP})^{*}(\chi _{p-1})    & \textup{if\ \ } G=\wP,\cr
(\pi_{\wwH _t})^{*}((\tau ' )^{p-1})    & \textup{if\ \ } G=\wwH_t,\cr
(\pi_{S^1})^{*}(\tau ^{p-1})        & \textup{if\ \ } G=S^1
\end{cases} \text{ and } $$
where $\pi _G:B_G\to BG$. Note that $z' _G$ is a primitive element and generates a $\bZ$ component in $H^{2p-2}(B_G;\bZ )$, where
$$H^{2p-2}(B_G;\bZ ) =
\begin{cases} 
 (\bZ /p )^{\oplus p} \oplus \bZ       & \textup{if\ \ } G=\wP,\cr
 (\bZ /p )^{\oplus p} \oplus \bZ        & \textup{if\ \ } G=\wwH_t,\cr
 \bZ         & \textup{if\ \ } G=S^1
\end{cases} $$
Define the \emph{cohomology fundamental class} $Z _G$ as follows:
$$Z _G = z_G \cup z' _G \in H^{4p-3}(B_G,\bZ )\ .$$ 
The reason for this definition will become clear after Lemma \ref{SpectralSequenceCalculationsLemma0}. In the spectral sequence for $H^*(B_G;\bZ)$, $G = \wwH_t$, the cohomology fundamental class of the manifold
$M_t$ lies in the term $E_{\infty }^{2p-2,2p-1}(G)$. In the formula for this term we see the elements $z_G$ and $z'_G$ described above.

\subsection{Transfers and duality}  \label{Tansfers}

In this section we will see the duality between  
$Z _{S^1}$, $Z _{\wwH _t}$, and $Z_{\wP}$
and $\gamma _{S^1}$, $\gamma _{\wwH _t}$, and elements in $T_{\wP}$
respectively. We will consider the $p$-fold covering maps
$$B_{\wwH _t}\to B_{\wP}\text{ , \ \ } B_{S^1}\to B_{\wwH _t} \text{ , \ \ }B\wwH _t\to B\wP\text{ , and \ \  } BS^1\to B\wwH _t$$
We will  write $\pi ^*$ and $\pi _*$ to denote the natural maps induced in cohomology and homology respectively, and just write $tr$ for the transfer maps both in cohomology and homology. We have 
$$ \pi ^*(z'_{\wwH_t})=z'_{S^1} \text {\ \ \ and \ \ \ }tr(z_{S^1})=z_{\wwH_t}$$
and 
$$ \pi ^*(z_{\wP})=z_{\wwH _t}\text {\ \ \ and \ \ \ }tr(z'_{\wwH_t})=z'_{\wP}$$
Hence we have
$$tr(Z _{S^1})=Z _{\wwH _t}\text{\ \ \  and \ \ \ }tr(Z _{\wwH _t})=Z _{\wP}$$
Note that $\CP^{p-1} \times S^{2p-1} $ is the universal covering of $M_{t}$ and we have the following pull-back diagram. 
$$\xymatrix@R+1pt{ \CP^{p-1} \times S^{2p-1}\ar[r]^(0.64)c\ar[d]& B_{S^1}\ar[d]\\ 
M_{t} \ar[r]^{c_t} & B_{\wwH_{t}}}$$
Hence we have
$$\tr(\gamma _{\wwH_{t}})=\gamma _{S^1}\ .$$
Considering the map $c\colon \CP^{p-1} \times S^{2p-1}\to B_{S^1}$ we have $$c^{*}(Z _{S^1})=A\times B$$ where $A$ is 
the cohomology fundamental class of $\CP^{p-1}$ and $B$ is the cohomology fundamental class of $S^{2p-1}$. This proves that $Z _{S^1}$ is a primitive element in $H^{4p-3}(B_{S^1};\bZ )$. Moreover 
$$\la Z _{S^1}, \gamma _{S^1} \ra=\la Z _{S^1}, c_{*}([\CP^{p-1} \times S^{2p-1}])  \ra=\la A \times B, [\CP^{p-1}] \times [S^{2p-1}]  \ra = 1\ .$$ 
Hence we have
$$\la Z _{\wwH_{t}}, \gamma _{\wwH_{t}}\ra=\la \tr(Z _{S^1}), \gamma _{\wwH_{t}}\ra=\la Z _{S^1}, \tr(\gamma _{\wwH_{t}}) \ra=1$$
and $Z_{\wwH_{t}}$ is a primitive element in $H^{4p-3}(B_{\wwH_{t}};\bZ )$.
Moreover $\gamma_{\wwH_{t}}$, and $\gamma_{S^1}$ are primitive elements in 
$H_{4p-3}(B_{\wwH_{t}};\bZ )$ and $H_{4p-3}(B_{S^1};\bZ )$ respectively. Hence our main calculation is the following: 
Given $\gamma \in T_{\wP}$, we have
$$ \la Z_{\wP}, \gamma\ra= \la \tr(Z_{S^1}, \gamma\ra = \la Z_{S^1}, \tr(\gamma)\ra = \la Z_{S^1}, \gamma_{S^1}\ra = 1$$
since $\tr(\gamma) - \gamma_{S^1}$ is a torsion element.

\subsection{Some spectral sequence calculations} \label{SpectralSequenceCalculations}

\begin{lemma} \label{SpectralSequenceCalculationsLemma0}
For $G=S^1$, $\wwH_ t$, or $\wP $, the differential $d_{2p}\colon E_{2p}^{2p-2,2p-1}( G )\to E_{2p}^{4p-2,0}(G)$ 
is surjective and its kernel is given as follows:
$$E_{\infty }^{2p-2,2p-1}(G)=
\begin{cases} 
\la p z_{2} \cdot \chi _{p-1} \ra    & \textup{if\ \ } G=\wP,\cr
\la p z_{2} \cdot (\tau ' )^{p-1} \ra    & \textup{if\ \ } G=\wwH_t,\cr
\la z_{2} \cdot \tau ^{p-1} \ra   & \textup{if\ \ } G=S^1
\end{cases} $$
\end{lemma}
\begin{proof}
We have $E_{2p}^{2p-2,2p-1}(S^1)= \la  z_{1} \cdot \tau ^{p-1},z_{2} \cdot \tau ^{p-1} \ra =\bZ ^{\oplus 2} $ 
and $ E_{2p}^{4p-2,0}(S^1)=\la  \tau ^{2p-1} \ra=\bZ$.
So result follows for $G=S^1$ because $d_{2p}( z_{1} \cdot \tau ^{p-1})=\tau ^{2p-1} $ spans $E_{2p}^{4p-2,0}(S^1)$.
We have 
$$E_{2p}^{2p-2,2p-1}(\wwH_{t})=(\bZ /p )^{\oplus 2p} \oplus \bZ ^{\oplus 2}$$ 
given by 
$$ \la  z_{1},z_{2} \ra  \cdot \la  (v ' ) ^{p-1}, (v ' ) ^{p-2}\tau '  , \dots,  (\tau ' )^{p-1} \ra \ .$$
and we have
$$E_{2p}^{4p-2,0}(\wwH_{t})= (\bZ /p )^{\oplus 2p+1}\oplus \bZ $$
given by
$$ \la (v ' ) ^{2p-1}, (v ' ) ^{2p-2}\tau ' , \dots , (\tau ' ) ^{2p-1} \ra \ .$$
The map 
$$d_{2p}\colon E_{2p}^{2p-2,2p-1}( \wwH_{t})\to E_{2p}^{4p-2,0}(\wwH_{t})$$ 
is surjective because the following list of images of $d_{2p}$ will span $E_{2p}^{4p-2,0}(\wwH_{t})$ considered as above:
\begin{itemize}\addtolength{\itemsep}{0.2\baselineskip}
\item $d_{2p}(z_{2}\cdot (v ' ) ^{s} (\tau ')^{p-1-s})=(v ' ) ^{p+s} (\tau ')^{p-1-s}$ for $0 \leq s \leq p-1$
\item $d_{2p}(z_{1}\cdot (v ' ) ^{s} (\tau ')^{p-1-s})=(v ' ) ^{s} (\tau ')^{2p-1-s} + (v ' ) ^{p-1+s} (\tau ')^{p-s}$ for $0 \leq s \leq p-1$
\end{itemize}  
This means we have $$E_{2p+1 }^{2p-2,2p-1}(\wwH_{t})=\la pz_{2}\cdot (\tau ')^{p-1}\ra $$ 
and the results follows for $G=\wwH_ t$,  since $\ker d_{2p}=\la pz_{2} \cdot (\tau ')^{p-1} \ra $. 
The proof for $G=\wP $ is left to reader as it will not be used in this paper.
\end{proof}

\begin{lemma} \label{SpectralSequenceCalculationsLemma1}
$H^{4p-3}(B_{S^1};\bZ )=\la Z_{S^1} \ra \oplus A$ where $A$ is a torsion group with no $p$-torsion.
\end{lemma}
\begin{proof}
By Lemma \ref{SpectralSequenceCalculationsLemma0}
$$F^{2p-1}H^{4p-3}(B_{S^1};\bZ )=0$$ 
and 
$$E_{\infty }^{2p-2,2p-1}(S^1)=\la z_{2} \cdot \tau ^{p-1}\ra \ .$$ 
It is clear that 
$$Z_{S^1}  \in  F^{2p-2}H^{4p-3}(B_{S^1};\bZ )$$
and represents the following generator of the quotient
$$z_{2} \cdot \tau ^{p-1}  \in E_{\infty}^{2p-2,2p-1}(S^1)=F^{2p-2}H^{4p-3}(B_{S^1};\bZ )/F^{2p-1}H^{4p-3}(B_{S^1};\bZ )\ .$$
Hence $F^{2p-2}H^{4p-3}(B_{S^1};\bZ )=\la Z_{S^1} \ra =\bZ $, and the fact $Z_{S^1}$ is a primitive element tells us that
$$H^{4p-3}(B_{S^1};\bZ )=\la Z_{S^1} \ra \oplus A$$
where $A$ is an abelian group. Now by Lemma \ref{CohomologyOfK}, 
we see that $E_{\infty}^{i,4p-3-i}( S^1)$ is a torsion group with no $p$-torsion for $0 \leq i \leq 2p-3$. Hence $A$ is a torsion group with no $p$-torsion.
\end{proof}

\begin{lemma} \label{SpectralSequenceCalculationsLemma2} 
For $G=S^1$ or $\wwH_ t$, we have $_{(p)}H^{4p-2}(B_{G};\bZ )=\bZ /p \oplus \bZ /p$.
\end{lemma} 
\begin{proof}
Note that $H^{2p-1}(B_{G};\bZ )=0$ hence by Lemma \ref{SpectralSequenceCalculationsLemma0} and Lemma \ref{CohomologyOfK}, one can see that $E_{\infty}^{r,4p-2-r}(G)$ has no $p$-torsion for $1\leq r \leq 4p-2$, and the $p$-torsion part of $E_{\infty}^{0,4p-2}(G)$ is $\bZ /p \oplus \bZ /p$.
\end{proof}

\begin{lemma} \label{SpectralSequenceCalculationsLemma3} 
The $p$-torsion subgroup of $H_{4p-3}(B_{\wwH_{t}}; \bZ ))$ is contained in the image of the natural map
$i_*\colon H_{4p-3}(K; \bZ ) \to H_{4p-3}(B_{\wwH_{t}}; \bZ ))$.
\end{lemma}
\begin{proof}
By duality and Lemma \ref{SpectralSequenceCalculationsLemma0}  we see that the torsion part of $H_{4p-3}(B_{\wwH_{t}}; \bZ ))$ is equal to
$F_{2p-3}H_{4p-3}(B_{\wwH_{t}};\bZ )$. Hence the results follows from Lemma \ref{HomologyOfK}.
\end{proof}

\begin{lemma} \label{SpectralSequenceCalculationsLemma4} 
The natural map $i_*\colon H_{2p-1}(K;\cy p) \to H_{2p-1}(B_{\wwH_t};\cy p)$ is zero.
\end{lemma}
\begin{proof}
It is clear that $i^*\colon H^{2p-1}(B_{\wwH_t};\cy p)\to H^{2p-1}(K;\cy p)$ is zero. Because both $\bar{z}_{1}$ and $\bar{z}_{2}$ in 
$$E^{0,2p-1}_{2}(\wwH_t,\cy p)=H^{2p-1}(K;\cy p)=\la \bar{z}_{1} ,\bar{z}_{2}\ra=\cy p \oplus \cy p$$ trangresses to nonzero elements in $E^{2p,0}_{2}(\wwH_t,\cy p)$.
\end{proof}

\subsection{The subset $T_{\wP}$ is non-empty}
Let $\tr\colon H_{4p-3}(B_{\wP}; \bZ )\to H_{4p-3}(B_{S^{1}}; \bZ )$ 
denote the transfer map.
\begin{lemma} \label{TransferFromHtoS}
Let $\gamma '$ in $H_{4p-3}(B_{\wP}; \bZ )$ be an element such that $\la Z _{S^1}, \tr(\gamma ')\ra=1$. 
Then there exists an integer $N_{\gamma '}$ such that $p(1-p^{2}N_{\gamma '})(\tr(\gamma ')-\gamma _{S^1})=0$. 
\end{lemma} 
\begin{proof} 
First note that $\la Z _{S^1}, \tr(\gamma ')-\gamma _{S^1}\ra=0$ since $\la Z _{S^1}, \gamma _{S^1}\ra=1$. Hence
$\tr(\gamma ')-\gamma _{S^1}$ is a torsion element, by the Universal Coefficient Theorem and Lemma \ref{SpectralSequenceCalculationsLemma1}. Hence it is enough to prove that the order of $p\cdot (\tr(\gamma ')-\gamma _{S^1})$ is relatively prime to $p$. But this is clear as the $p$-torsion of part of $H_{4p-3}(B_{S^{1}}; \bZ )$ is same as the $p$-torsion part of $H^{4p-2}(B_{S^{1}}; \bZ )$, which is $\bZ /p \oplus \bZ /p$ by Lemma \ref{SpectralSequenceCalculationsLemma1}.
\end{proof}

\begin{theorem}
\label{ExistenceOfPrimitiveElementInHomologyLemma} The set $T_{\wP}$ is not empty. Any $\gamma \in T_{\wP}$ is a primitive element of infinite order in $H_{4p-3}(B_{\wP}; \bZ )$. 
\end{theorem} 
\begin{proof} 
Let $\tr$ denote the (co)homology transfer associated to the covering map 
$B_{S^1} \xrightarrow{\pi } B_{\wP}\ .$ 
We know that $\tr(Z_{S^1})=Z_{\wP}$ and $Z_{\wP}$ is a primitive element in $H^{4p-3}(B_{\wP };\bZ )$. Hence by the Universal Coefficient Theorem there exists a primitive element $\gamma '$ in $H_{4p-3}(B_{\wP}; \bZ )$ such that $\tr(\gamma ')$ is a primitive element in $H_{4p-3}(B_{S^1}; \bZ )$ and 
$$\la Z_{\wP}, \gamma '\ra=1 \text{  and  }\la Z_{S^1}, \tr(\gamma ')\ra=1\ .$$ 
Take $N_{\gamma '}$ as in Lemma \ref{TransferFromHtoS}. 
Define $$\gamma _{\wP}= \gamma ' - N_{\gamma '}(\pi )_{*}(\tr(\gamma ')-\gamma _{S^1})\ .$$
Then $ \gamma _{\wP}$ is in $T_{\wP}$, because by Lemma \ref{TransferFromHtoS} we have
$$ p\cdot (\tr(\gamma _{\wP})-\gamma _{S^1} )=p(\tr(\gamma ' - N_{\gamma '}(\pi )_{*}(\tr(\gamma ')-\gamma _{S^1}) ) -\gamma _{S^1} )=p(1-p^{2}N_{\gamma '})(\tr(\gamma ')-\gamma _{S^1})=0$$ 
Now, take any $\gamma $ in $T_{\wP}$. Suppose that $\gamma = r\gamma_1$, for some  $\gamma_1$ in $H_{4p-3}(B_{\wP}; \bZ )$. Then we would have $p\cdot(\gamma _{S^{1}}- r\cdot \tr(\gamma_1 ))=0$.  But $\la Z_{S^1}, \gamma _{S^{1}}\ra = 1$. Hence $r=\pm 1$.
\end{proof}

\begin{proposition}
\label{HomologyOfB_H} Let $\tr\colon H_{4p-3}(B_{\wP}; \bZ )\to H_{4p-3}(B_{\wwH_{t}}; \bZ )$ denote the transfer map. Then any $\gamma$ in $T_{\wP}$ satisfies the following equation $ p(\tr(\gamma) - \gamma _{\wwH_{t}})=0 $.
\end{proposition}
\begin{proof}
For any $\gamma $ in $T_{\wP} $ the image of $p\cdot (\tr(\gamma )-\gamma _{\wwH_{t}})$ under the transfer map from $H_{4p-3}(B_{\wwH_{t}}; \bZ )$ to $H_{4p-3}(B_{S^1}; \bZ )$ is $0$, by definition of $T_{\wP} $ and the fact that $\tr(\gamma _{\wwH_{t}})=\gamma _{S^1}$. Note that the kernel of the above transfer map is included in the $p$-torsion part of $H_{4p-3}(B_{\wwH_{t}}; \bZ )$, as $B_{S^1}\to B_{\wwH_{t}}$ is $p$-covering. By Lemma \ref{SpectralSequenceCalculationsLemma1}, the  $p$-torsion part of $H_{4p-3}(B_{\wwH_{t}}; \bZ )$ is $\bZ /p \oplus \bZ /p$ (which has exponent $p$). This proves the result.
\end{proof}

\section{The construction of the bordism element}\label{sec: bordism element}
The next step in our argument is to study the bordism groups 
$\Omega _{4p-3}(B_{\wP, }\nu _{\wP})$ of our normal $(2p-1)$-type. The main result of this section is Theorem \ref{hurewicz}, which proves that the image of
the Hurewicz map 
$$\Omega _{4p-3}(B_{\wP, }\nu _{\wP})\to
H_{4p-3}(B_{\wP};\bZ  )$$
contains the non-empty subset $T_{\wP}$ (see Definition \ref{def: define TsubGamma}). The main difficulty in computing the bordism groups is dealing with $p$-torsion. We will primarily use the James spectral sequence (a variant of the Atiyah-Hirzebruch spectral sequence) associated to the fibration 
$$ * \longrightarrow B_{\wP} \longrightarrow B_{\wP} $$ 
with $E^2$-term
$$E^2_{n,m}(\nu_{\wP}) = H_n(B_{\wP}; \Omega^{fr}_m(\ast)),$$
where the coefficients $\Omega^{fr}_m(\ast)= \pi_m^S$ are the stable homotopy groups of spheres. In our range, the $p$-torsion in $\pi_m^S$ occurs only for 
 $\pi^S_{2p-3}$ and $\pi^S_{4p-5}$, where the $p$-primary part is $\cy p$ (see \cite[p.~5]{ravenel1} and Example \ref{StableStem}). This means that, after localizing at $p$, there are only two possibly non-zero differentials with source at the
 $(4p-3,0)$-position, namely $d_{2p-2}$ and $d_{4p-4}$. To show that these differentials are in fact both zero, and to prove that all other differentials starting at the $(4p-3,0)$-position also vanish on $T_{\wP}$, we use two techniques:
 \begin{enumerate}\addtolength{\itemsep}{0.1\baselineskip}
\renewcommand{\labelenumi}{(\roman{enumi})}
 \item For the differentials $d^r$ with $2\leq r \leq 4p-5$, and $d^{4p-3}$, we compare the James spectral sequence for $\Omega _{4p-3}(B_{\wP, }\nu _{\wP})$ to the ones for $\Omega _{4p-3}(B_{\wwH_t},\nu _{\wwH_t})$ via transfer, and use naturality.
 
 \item For the differential $d^{4p-4}$ we compare the 
 James spectral sequence for $\Omega _{4p-3}(B_{\wP, }\nu _{\wP})$ to the James spectral sequences for the fibrations
 $B\wwH_{t} \to B\wP  \to B(\wP/\wwH_{t})$, and use naturality again. 
\end{enumerate}
In carrying out the second step, we will need to use the Adams spectral sequence to prove that the natural map
from the $p$--component of $\Omega ^{fr}_{4p-5}(*)$ to $\Omega _{4p-5}(B\wwH_{t}, \Bnu{\wwH_{t}})$ is injective (see Theorem
\ref{NatInjective}).

\subsection{The James Spectral Sequence}
Let $\{E_{n,m}^{r}(\nu )\}$ denote the James spectral sequence (see \cite{teichner1}) associated to a vector bundle $\nu $ over a base space $B$ and the fibration 
$$ * \longrightarrow B \longrightarrow B $$ 
and denote the differentials of this spectral sequence by $d^{r}$. We know that the second page is given by 
\begin{equation*}
E_{n,m}^{2}(\nu )=H_{n}(B,\Omega _{m}^{fr}(\ast ))
\end{equation*} 
and the spectral sequence converges to
\begin{equation*}
E_{n,m}^{\infty }(\nu )=F_{n}\Omega _{n+m}(B, \nu )/F_{n-1}\Omega _{n+m}(B, \nu )
\end{equation*}
where $B^{(n)}$ stands for the $n^{th}$ skeleton of $B$ and 
\begin{equation*}
F_{n}\Omega _{n+m}(B, \nu )=\Image ( \Omega _{n+m}( B^{(n)}, \nu |_{B^{(n)}})\to \Omega _{n+m}(B, \nu ))
\end{equation*}
For $0 \leq t \leq p $, let $$\tr_{t}\colon E_{n,m}^{r}(\nu _{\wP})\to E_{n,m}^{r}(\nu _{\wwH_{t}})$$ denote the
transfer map.

\subsection {Calculation of $d^{r}$ when $2 \leq r \leq 4p-5$}
Here we employ our first technique. We first need some information about the James spectral sequences for $\Omega _{4p-3}(B_{\wwH_t, }\nu _{\wwH_t})$.
\begin{lemma}
\label{LemmaForLessThan4p-4forH} For $2 \leq r \leq 4p-5$, the differential   
$$d^{r}\colon E_{4p-3,0}^{r}(\nu _{\wwH_{t}})\to E_{4p-3-r,r-1}^{r}(\nu _{\wwH_{t}})$$
is zero on $\tr_{t}(T_{\wP})$, where $T_{\wP}$ is considered as subgroup of $E_{4p-3,0}^{r}(\nu _{\wP})$.
\end{lemma}
\begin{proof}
Assume $2 \leq r \leq 4p-5$. By Proposition \ref{HomologyOfB_H}, and the fact that $d^{r}(\gamma _{\wwH_{t}})=0$, it is enough to show that   
$d^{r}\colon E_{4p-3,0}^{r}(\nu _{\wwH_{t}})\to E_{4p-3-r,r-1}^{r}(\nu _{\wwH_{t}})$
is zero on the $p$-torsion subgroup $I_{t}$ of $E_{4p-3,0}^{r}(\nu _{\wwH_{t}})=H_{4p-3}(B_{\wwH_{t}}; \bZ )$. Let $I$ be the $p$-torsion part of 
$E_{4p-3,0}^{r}(\nu _{\wwH_{t}} | _{K})=H_{4p-3}(K; \bZ )$. By Lemma \ref{SpectralSequenceCalculationsLemma3}, we have
$$I_{t}= \Image \{ i_{*}\colon I \to H_{4p-3}(B_{\wwH_{t}}; \bZ )\} \ .$$
We consider the cases  (i) $2 \leq r \leq 2p-3$, (ii) $r= 2p-2$, and (iii) $2p-1 \leq r \leq 4p-5$ separately.

\medskip
\noindent
Case (i). The group $E_{4p-3-r,r-1}^{2}(\nu _{\wwH_{t}} | _{K})$ is $p$-torsion free for $2 \leq r \leq 2p-3$, it follows that the 
 differential 
$$d^{r}\colon E_{4p-3,0}^{r}(\nu _{\wwH_{t}}  | _{K}  )\to E_{4p-3-r,r-1}^{r}(\nu _{\wwH_{t}} | _{K}  )$$
is zero on $I$. Hence the differential   
$$d^{r}\colon E_{4p-3,0}^{r}(\nu _{\wwH_{t}})\to E_{4p-3-r,r-1}^{r}(\nu _{\wwH_{t}})$$
is zero on $I_{t}$, for $2 \leq r \leq 2p-3$.

\medskip
\noindent
Case (ii).
Next we observe that the map
$i_{*}\colon E_{2p-1,2p-3}^{2}(\nu _{\wwH_{t}} | _{K})\to E_{2p-1,2p-3}^{2}(\nu _{\wwH_{t}})$, restricted to $p$-torsion, is just the natural map $i_*\colon H_{2p-1}(K;\cy p) \to H_{2p-1}(B_{\wwH_t};\cy p)$, which  is zero by Lemma \ref{SpectralSequenceCalculationsLemma4} 
Hence, the differential   
$$d^{2p-2}\colon E_{4p-3,0}^{2p-2}(\nu _{\wwH_{t}})\to E_{2p-1,2p-3}^{2p-2}(\nu _{\wwH_{t}})$$
is zero on $I_{t}$ by naturality.

\medskip
\noindent
Case (iii).
Finally, we note that
 $I_{t}$ is all $p$--torsion, but for $2p-1 \leq r \leq 4p-5$, the group $E_{4p-3-r,r-1}^{2}(\nu _{\wwH_{t}})$ is $p$--torsion free. 
 Hence, for $2p-1 \leq r \leq 4p-5$, the differential   
$$d^{r}\colon E_{4p-3,0}^{r}(\nu _{\wwH_{t}})\to E_{4p-3-r,r-1}^{r}(\nu _{\wwH_{t}})$$
is zero on $I_{t}$.
\end{proof}

\begin{lemma}
\label{LemmaForLessThan4p-4}For $2 \leq r \leq 4p-5$, the differential
 $$d^{r}\colon E_{4p-3,0}^{r}(\nu _{\wP})\to E_{4p-3-r,r-1}^{r}(\nu _{\wP})$$
 is zero on $T_{\wP}$, where $T_{\wP}$ is considered as a subgroup of $E_{4p-3,0}^{r}(\nu _{\wP})$.
\end{lemma}
\begin{proof} Assume $2 \leq r \leq 4p-5$. By Lemma \ref{LemmaForLessThan4p-4forH} we know that $$d^{r}\colon E_{4p-3,0}^{r}(\nu _{\wwH_{t}})\to E_{4p-3-r,r-1}^{r}(\nu _{ \wwH_{t}})$$
is zero for all $t \in \{0,1,\dots,p\} $. Hence it is enough to show that 
$$\bop_{t}\tr_{t}\colon E_{4p-3-r,r-1}^{r}(\wP)\to \bop\nolimits_{t}E_{4p-3-r,r-1}^{r}(\wwH_{t})$$
is injective. The map $\tr_{0}$
is clearly injective for $r\neq 2p-2$ because the $p$--component of $\Omega _{r-1}^{\ast }(\mathbb{\ast })$ is $0$ and $B_{ 
\wwH_{t}}\to B_{\wP}$ is a $p$-covering
map. Hence $\bop\nolimits_{t}\tr_{t}$ is injective for $r\neq 2p-2$. Now we know that $\widetilde{B_{\wP}}$ and $\widetilde{B_{ \wwH_{t}}}$ are $2p-2$ connected. Hence for $r=2p-2$ the map $\tr_{t}$ is
the usual transfer map $H_{2p-1}(B\wP;\bZ   
/p)\to H_{2p-1}(B\wwH_{t};\bZ  /p)$. Hence, it is
enough to show that the map 
$$\bop\nolimits_{t}\tr_{t}\colon H_{2p-1}(B\wP ;\bZ  /p)\to \bop\nolimits_{t}H_{2p-1}(B\wwH_{t};\bZ  /p)$$
 is injective. Dually, this is equivalent to showing that 
$$\bop\nolimits_{t} \tr_{t}\colon \bop\nolimits_{t} H^{2p-1}(B\wwH_{t};\bZ   
/p)\to H^{2p-1}(B\wP;\bZ  /p)$$ is surjective.
By Theorem \ref{TheoremofLearymodp}, we know that $$H^{2p-1}(B\wP;\bZ  /p)=\la x^{p-1}y,x^{p-2}x^{\prime }y,\dots,(x^{\prime
})^{p-1}y,(x^{\prime })^{p-1}y^{\prime }\ra\ .$$
Under the Bockstein homomorphism, this  can be identified with
$$V_{p+1} = \la \alpha^p, \alpha^{p-1}\beta, \dots, \alpha\beta^{p-1}, \beta^p\ra \subseteq H^{2p}(B\wP;\bZ)$$
and this identifcation is natural with respect to the action of the automorphisms
$\Aut(\wP)$ acting through the induced map $\Aut(\wP) \to GL_2(p)$. 
In the statements of Theorem 1 and Theorem 2 of [25], Leary gives explicit formulas for the action of $\Aut(\wP)$ on the generators of the cohomology rings
$H^*(B\wP;\bZ)$  and $H^*(B\wP;\bZ/p)$. The point is that these cohomology generators are pulled back from the quotient group $\cy p \times \cy p$. 

Hence the action of the automorphisms $\phi_A \in \Aut(\wP)$,  defined in Remark \ref{rem: lifting automorphisms} for all  $A \in SL_2(p)$,  gives the standard $SL_2(p)$-action on $V_{p+1}$.
 This module $V_{p+1}$ is known to be an indecomposable $SL_2(p)$-module (see \cite[5.7]{glover1}), and there is a short exact sequence
$$0 \to V_2 \to V_{p+1} \to V_{p-1} \to 0$$
of $SL_2(p)$-modules, where $V_2= \la \alpha^p, \beta^p\ra$ has dimension 2 and $V_{p-1}$ is irreducible. 

 Now the image of the map $ 
\bop\nolimits_{t}\tr_{t}$ is invariant under all automorphisms of
the group $\wP$. Hence it is enough to show that
$\Image(\bop\nolimits_{t}\tr_{t})$ projects non-trivially into $V_{p-1}$.
However, the calculations of \cite[p.~67]{leary2} show that 
$$\tr_p(\Res_{\wwH_p}(y')\cdot \bar\tau^{p-1}) = y'\cdot \tr_p(\bar\tau^{p-1}) = y'(c_{p-1} + x^{p-1}) = -(x')^{p-1}y' + y'x^{p-1}\ . $$
After applying the Bockstein, this shows that the element $\beta^p - \beta\alpha^{p-1}$ is contained in the image of the transfer. Since this element is not contained in the submodule $V_2$, we are done.
\end{proof}

\medskip
The remaining possibly non-zero differentials are $d^{4p-4}$ and $d^{4p-3}$. The first one is handled by comparison with the fibrations
$$B\wwH_{t} \to B\wP  \to B(\wP/\wwH_{t})$$
but first we must show that the induced map on coefficients at the $(0,4p-3)$-position is injective on the $p$-component. For this we use the Adams spectral sequence.
\newline

\subsection{The Adams spectral sequence}\label{SectionOnAdamsSpecSeqn} 
Let $X$ be a connective spectrum of finite type. We will write 
$$X=\{X_{n},i_{n}\vv {n \geq 0}\},$$
where each $X_{n}$ is a space with a basepoint and  
$i_{n}\colon \Sigma X_{n}\to  X_{n+1}$ is a basepoint preserving map.  
We will denote the Adams spectral sequnce for $X$ as follows:
$$\{E^{n,m}_{r}(X),d_{r}\}\ .$$
The second page of this spectral sequence is given by 
$$  E^{n,m}_{2}(X) = \Ext_{\cA_p}^{n,m}(H^{*}(X; \bZ /p),\bZ /p), $$
where $\cA_p$ is the mod-$p$ Steenrod algebra and $H^{*}(X; \bZ /p)$ is considered as an $\cA_p$-module. The differentials of this spectral sequence are as follows:
$$  d_{r}\colon E^{n,m}_{r}\to E^{n+r,m+r-1}_{r} $$
for $r\geq 2$, and it converges to 
$$ _{(p)} \pi ^{S}_{*}(X) = \pi ^{S}_{*}(X) / \la \text{torsion prime to } p \ra $$
with the filtration 
$$\dots\subseteq  F^{2,*+2}(X) \subseteq F^{1,*+1}(X) \subseteq F^{0,*}(X) =
  \hphantom{}_{(p)}\pi ^{S}_{*}(X)$$ 
defined by:
$$ F^{n,m}(X)=  \hphantom{}_{(p)}\Image \{  \pi ^{S}_{m}(X_{n})\to  \pi ^{S}_{m-n}(X) \}$$
In other words,
$$ E^{n,m}_{\infty }(X) = F^{n,m}(X)/F^{n+1,m+1}(X)\ .$$
\begin{example}\label{StableStem}
Take an $\cA_p$--free resolution $F^{\bbS}_*$ of the sphere spectrum $\bbS $$$ \dots \xrightarrow{\partial _{3}}  F^{\bbS }_{2}
\xrightarrow{\partial _{2}}  F^{\bbS }_{1}
\xrightarrow{\partial _{1}}  F^{\bbS }_{0} 
\xrightarrow{\partial _{0}}  H^{*}(\bbS ; \bZ /p) $$
with the following properties: 
\begin{itemize}
\item We have $\iota ^{\bbS}_{0}$ in $F^{\bbS }_{0}$, such that $\partial _{0}(\iota ^{\bbS}_{0})\text{ is a generator of }H^{*}(\bbS; \cy p)=\cy p$. \smallskip
\item We have $\alpha ^{\bbS}_{0}$ and $\alpha ^{\bbS}_{2p-3}$ in $F^{\bbS }_{1}$, such that 
$\partial _{1}(\alpha ^{\bbS}_{0})=\beta(\iota ^{\bbS}_{0})$ and $\partial _{1}(\alpha ^{\bbS}_{2p-3})=P^{1}(\iota ^{\bbS}_{0})$ because 
$H^{i}(\bbS ; \bZ /p)=0 \text{\  for\  }i\geq 1$. \smallskip
\item We have $\beta ^{\bbS}_{4p-5}$ in $F^{\bbS }_{2}$, where $\partial _{2}(\beta ^{\bbS}_{4p-5})=P^2(\alpha ^{\bbS}_{0})-P^{1}\beta (\alpha ^{\bbS}_{2p-3})+2\beta P^{1}(\alpha ^{\bbS}_{2p-3})$ because 
$$ P^2(\beta(\iota ^{\bbS}_{0}))-P^{1}\beta (P^{1}(\iota ^{\bbS}_{0}))+2\beta P^{1}(P^{1}(\iota ^{\bbS}_{0}))=0 \ .$$
\end{itemize}
In the Adams spectral sequence that converges to the $p$--component of $\pi ^{S}_{*}(\bbS )=\Omega ^{fr}_{*}(*)$, the element $\beta ^{\bbS}_{4p-5}$ must survive to the $E^{\infty }$--term as there are no possible differentials. Hence we have the following:
\begin{enumerate} 
\item $_{(p)}\Omega ^{fr}_{2p-3}(*)=\bZ /p =\la \alpha ^{\bbS}_{2p-3} \ra$
\item $_{(p)}\Omega ^{fr}_{4p-5}(*)=\bZ /p =\la \beta ^{\bbS}_{4p-5} \ra$
\end{enumerate}
\end{example}

\subsection{Cohomology of the Thom spectrum associated to $\Bnu{G}$ } \label{SectionOnModuleStructure}
Now take any $G\subseteq \wP$ and let $M\Bnu{G}$ denote the Thom spectrum associated to the bundle $\Bnu{G}$. Since the bundle $\xi_G$ is fixed, for a given $G$, we will shorten the notation by writing $\MS G = M\Bnu{G}$. As in the previous section, we will denote an $\cA_p$--free resolution of $H^{*}(\MS G;\bZ /p)$ as follows:
\begin{equation*}
\dots \xrightarrow{\partial _{2}}  
F^{\MS G}_{1} \xrightarrow{\partial _{1}}  
F^{\MS G}_{0} \xrightarrow{\partial _{0}}   H^{*}(\MS G; \bZ /p) 
\end{equation*}
It is clear that, to understand these resolutions we must first understand the $\cA_p$--module structure on the cohomology $H^{*}(\MS G; \bZ /p)$ of these spectra.  
Let $U_{G} \in H^0(\MS G; \bZ /p)$ denote the Thom class of the Thom spectrum $\MS G$. Then we can write 
$$H^{*}(\MS G;\bZ /p) = U_{G} \cdot H^{*}(BG;\bZ /p)\ .$$
Moreover, for $G=S^1$ we will write 
$$ H^{*}(BS ^{1};\bZ /p)= \bF_p[\bar{\tau }], $$
and for $G=D_{t}$ we have
$$ H^{*}(BD_{t};\bZ /p) = ( \Lambda (u) \otimes \bF_p [v] )\ .   $$
Hence for $G=\wwH_{t}$ we can consider 
$$ H^{*}(B\wwH_{t};\bZ /p)=H^{*}(BD_{t};\bZ /p) \otimes H^{*}(BS ^{1};\bZ /p)
= ( \Lambda (u) \otimes \bF_p [v] ) \otimes \bF_p[\bar{\tau }]$$ 
\begin{lemma} 
For $G=S^1$, $D_{t}$, or $\wwH_{t}$ we have
$$ \beta (U_{G})=0,\quad P^1 \beta (U_{G})=0, \quad \beta P^1 (U_{G})=0, \quad \beta P^1 \beta (U_{G})=0, \quad P^2 (U_{G})=0$$ 
and
$$ P^1 (U_{G})= 
\begin{cases} 
U_{\wwH_{t}}v^{p-1}   & \textup{if\ \ } G=\wwH_t,\cr
U_{D_{t}}v^{p-1}   & \textup{if\ \ } G=D_t,\cr
0    & \textup{if\ \ } G=S^1
\end{cases} $$
\end{lemma}
\begin{proof} 
The Thom class $U_{G}$ is the mod $p$ reduction of an integral cohomology class, so $\beta (U_{G})=0$.
 By Lemma \ref{sphericalclass},  
$q_{1}(\Bnu{\wwH_{t}})=v^{p-1}$. 
Since 
$P^1(U_{\wwH_{t}})=U_{\wwH_{t}}v^{p-1}$, we obtain
$$ P^1 (U_{S^1})=0 \text{\ \  and \ \ } P^1 (U_{D_{t}})=U_{D_{t}}v^{p-1} $$ 
by restriction  to $H^{*}(BD_{t};\bZ /p)$ and $H^{*}(BS ^{1};\bZ /p)$.
For $G=D_{t}$ or $\wwH_{t}$ we have
$$\beta P^1 (U_{G})=\beta (U_{G}v^{p-1})=\beta (U)v^{p-1}+U\beta (v^{p-1})=0+0=0$$
and it is clear that
$\beta P^1 (U_{S^1})=0$.
 By the Adem relations we have $P^2 (U_{G}) = 2P^1P^1(U_{G})$.
Hence for $G=D_{t}$ or $\wwH_{t}$ we have
$$P^2 (U_{G})=2P^1(U_{G}v^{p-1})=2(P^1(U_{G})v^{p-1})+U_{G}P^1(v^{p-1}))=2(U_{G}v^{2p-2}-U_{G}v^{2p-2})=0$$
and it is clear that
$P^2 (U_{S^1})=0$.
\end{proof}
\subsection{Calculation of $d^{4p-4}$}
The inclusion of a point induces a
natural map from $\Omega ^{fr}_{4p-5}(*)$ to $\Omega _{4p-5}(B\wwH_{t}, \Bnu{\wwH_{t}})$ for each of the subgroups $\wwH_t$, $ 0\leq t \leq p$.
\begin{theorem}\label{NatInjective}
The natural map  $\Omega ^{fr}_{4p-5}(*)\to \Omega _{4p-5}(B\wwH_{t}, \Bnu{\wwH_{t}})$ is injective on the $p$--component.
\end{theorem}
\begin{proof} The generator of $_{(p)}\Omega ^{fr}_{4p-5}(*)$ is represented by the class $\beta_{4p-5}^{\bbS}$ defined above. We will show that this element maps non-trivially in the Adams spectral sequence. Denote the elements of $H^{*}(\MS{\wwH_{t}}; \bZ /p)$, $H^{*}(\MS{S^1}; \bZ /p)$, and $H^{*}(\MS{D_{t}}; \bZ /p)$ as in Section \ref{SectionOnModuleStructure}. It is straightforward to check the following:

\nr We have $\iota ^{\MS{\wwH_{t}}}_{0}$ and $\iota ^{\MS{\wwH_{t}}}_{2p-3}$ in $F^{\MS{\wwH_{t}}}_0$ such that 
$$\partial _{0}(\iota ^{\MS{\wwH_{t}}}_{0})=U_{\wwH_{t}}\text{\ \ \ and\ \ \ }\partial _{0}(\iota ^{\MS{\wwH_{t}}}_{2p-3})=U_{\wwH_{t}}uv^{(p-2)}\ .$$

\nr We have $\alpha ^{\MS{\wwH_{t}}}_{0}$ and $\alpha ^{\MS{\wwH_{t}}}_{2p-3}$ in $F^{\MS{\wwH_{t}}}_1$ such that
$$\partial _{1}(\alpha ^{\MS{\wwH_{t}}}_{0})=\beta(\iota ^{\MS{\wwH_{t}}}_{0})\text{ \  and \  }\partial _{1}(\alpha ^{\MS{\wwH_{t}}}_{2p-3})=P^1(\iota ^{\MS{\wwH_{t}}}_{0})- \beta(\iota ^{\MS{\wwH_{t}}}_{2p-3})$$ 
because $ \beta(U_{\wwH_{t}})=0$ and $P^1(U_{\wwH_{t}})- \beta(Uuv^{(p-2)})=0 $. 

\nr We also have $\alpha ^{\MS{\wwH_{t}}}_{4p-5}$ in $F^{\MS{\wwH_{t}}}_1$ such that $$\partial _{1}(\alpha ^{\MS{\wwH_{t}}}_{4p-5})=P^1 \beta (\iota ^{\MS{\wwH_{t}}}_{2p-3})$$ because
$P^1(\beta(U_{\wwH_{t}}uv^{(p-2)}))=0$. 

\nr We have $\beta ^{\MS{\wwH_{t}}}_{4p-5}$ in  $F^{\MS{\wwH_{t}}}_2$ such that
$$\partial _{2}(\beta ^{ \MS{\wwH_{t}}}_{4p-5})=P^2(\alpha ^{ \MS{\wwH_{t}}}_{0})-P^{1}\beta (\alpha ^{ \MS{\wwH_{t}}}_{2p-3})+2\beta P^{1}(\alpha ^{ \MS{\wwH_{t}}}_{2p-3})+2\beta (\alpha ^{ \MS{\wwH_{t}}}_{4p-5})$$ because
$$P^2( \beta(\iota ^{\MS{\wwH_{t}}}_{0})  )-P^{1}\beta (P^1(\iota ^{\MS{\wwH_{t}}}_{0})- \beta(\iota ^{\MS{\wwH_{t}}}_{2p-3}))+2\beta P^{1}(P^1(\iota ^{\MS{\wwH_{t}}}_{0})- \beta(\iota ^{\MS{\wwH_{t}}}_{2p-3}))+2\beta P^1 \beta (\iota ^{\MS{\wwH_{t}}}_{2p-3})=0\ .$$
Now we define a part of the chain map $ F^{ \MS{\wwH_{t}}}_{*} \to F^{\bbS }_{*}$. We send
$$\iota ^{ \MS{\wwH_{t}}}_{0} \mapsto \iota ^{\bbS }_{0} \text{\ \  and \ \ } \iota ^{ \MS{\wwH_{t}}}_{2p-3}\mapsto 0\ . $$
Since $\beta (\iota ^{\MS{\wwH_{t}}}_{0}) \mapsto \beta(\iota ^{\bbS }_{0})$ 
and $P^1(\iota ^{\MS{\wwH_{t}}}_{0})- \beta(\iota ^{\MS{\wwH_{t}}}_{2p-3}) \mapsto P^1(\iota ^{\bbS }_{0})$ we must have
$$\alpha ^{ \MS{\wwH_{t}}}_{2p-3} \mapsto \alpha ^{\bbS }_{2p-3}\text{\ \ and \ \ } \alpha ^{ \MS{\wwH_{t}}}_{4p-5}\mapsto 0 \ .$$
Finally, we can send
$$\beta ^{\MS{\wwH_{t}}}_{4p-5}\mapsto \beta ^{\bbS }_{4p-5}$$
and this definition proves the Theorem, as there are no differentials in this range.
\end{proof}

\begin{remark}\label{homotopysphere}
 A similar technique can be used to prove that
the natural map $\Omega_{10}^{fr}(\ast) \to \Omega_{10}(BS^1, \xi_{S^1})$
is injective on the $3$-component. One constructs a chain map
$F_*^{\bbS} \to F_*^{\MS{S^1}}$ in degrees $\leq 11$, whose composite with the chain map induced by the natural map $H^*(\MS{S^1}; \bZ /p) \to H^*(\bbS; \bZ /p)$ is chain homotopic to the identity.
The element $\beta^{\bbS}_{10}$ generating the $3$-component of $\pi^S_{10}$ arises from $P^2(\alpha_3^{\bbS})$ and the Adem relation $P^2P^1 \iota_0^{\bbS}= 0$.
\end{remark}

\begin{lemma}
\label{LemmaFor4p-4} $d^{4p-4}\colon E_{4p-3,0}^{4p-4}(\nu _{\wP})\to E_{1,4p-5}^{4p-4}(\nu _{\wP})$ is zero.
\end{lemma}
\begin{proof}
 We consider the  fibration 
\begin{equation*}
B\wwH_{t} \longrightarrow B\wP  \longrightarrow B(\wP/\wwH_{t}) 
\end{equation*}
for $0\leq t \leq p $.
This fibration induces a James spectral sequence $E_{*,*}^{*}(t)$ with differential denoted by $d^{*}_{t}$ so that the second page is given by 
\begin{equation*}
E_{n,m}^{2}(t)=H_{n}(\wP/\wwH_{t},\Omega _{m}(B{\wwH_{t}},\Bnu{\wwH_{t}}))
\end{equation*} 
and the spectral sequence converges to $\Omega _{*}(B{\wP}, \Bnu{ 
\wP} )$. Moreover, we have a natural map 
$ E_{*,*}^{4p-4}(\nu _{\wP}) \to E_{*,*}^{4p-4}(t) $ due to the following map of fibrations.
$$\xymatrix{\ast \ar[r]\ar[d]&B_{\wP}\ar[r]\ar[d]&B_{\wP}\ar[d]\cr
B\wwH_{t}\ar[r]&B\wP\ar[r]&B(\wP/\wwH_{t})
}$$
Theorem \ref{NatInjective} (applied for $t=0$ and $t=p$), and the detection of $H_1(\cy p\times \cy p; \bZ /p)$ by cyclic quotients,  shows that the following sum of two of these natural maps is injective 
\begin{equation*} 
 E_{1,4p-5}^{4p-4}(\nu _{\wP}) \to E_{1,4p-5}^{4p-4}(0) \oplus E_{1,4p-5}^{4p-4}(p) \ . 
\end{equation*}
However, the differential $d^{4p-4}_{t}\colon E_{4p-3,0}^{4p-4}(t) \to E_{1,4p-5}^{4p-4}(t)$ is zero for both 
$t=0$ and $t=p$, since the element $N_{p-t} \to BD_{p-t} \to B\wP$, for $t= 0,p$ (defined in Section \ref{mainexamples}) is non-zero in $\Omega _{4p-3}(B{\wP},\Bnu{ 
\wP} )$. This is because $[N_{p-t}]\in H_{4p-3}(BD_{p-t};\bZ)$ is non-zero, and
the inclusion $D_{p-t} \subset \wP$ is split on homology by projection to $\wP/\wwH_t\cong D_{p-t}$.
\end{proof}

\subsection{Calculation of $d^{4p-3}$}
The last differential doesn't involve $p$-torsion in the target, and can be handled by one more transfer argument. 
\begin{lemma}
\label{LemmaFor4p-3} $d^{4p-3}\colon E_{4p-3,0}^{4p-3}(\nu _{\wP})\to E_{0,4p-4}^{4p-4}(\nu _{\wP})$ is zero on $T_{\wP}$ where $T_{\wP}$
 is considered as a subgroup of $E_{4p-3,0}^{4p-3}(\nu _{\wP})$.
\end{lemma}
\begin{proof}
By Lemma \ref{LemmaFor4p-4} and the transfer map $\tr_{t}$ we see that the differential   
$$d^{4p-4}\colon E_{4p-3,0}^{4p-4}(\nu _{\wwH_{t}})\to E_{1,4p-5}^{4p-4}(\nu _{\wwH_{t}})$$
is zero on $\tr_{t}(T_{\wP})$ where $T_{\wP}$ is considered as subgroup of $E_{4p-3,0}^{4p-4}(\nu _{\wP})$.
Now the differential $$d^{4p-3}\colon E_{4p-3,0}^{4p-3}(\nu _{\wwH_{t}})\to E_{0,4p-4}^{4p-4}(\nu _{\wwH_{t}})$$ 
has to be zero on $\gamma _{\wwH_{t}}$ and on the $p$-torsion group $\Image \{H_{4p-3}(K) \to H_{4p-3}(B_{\wwH_{t}}; \bZ)\}$, by Proposition \ref{HomologyOfB_H} and the fact that 
$E_{0,4p-4}^{4p-4}(\nu _{\wwH_{t}})$ is $p$-torsion free: this term is a quotient of $H_0(B_{\wwH_{t}};\Omega^{fr}_{4p-4}(\ast)) \cong \pi^S_{4p-4}$, which has no $p$-torsion. Hence the result follows.
\end{proof}

We have now proved the main result of this section.
\begin{theorem}\label{hurewicz}
The subset $T_{\wP}\neq \emptyset$ is contained in the image of the Hurewicz map $\Omega _{4p-3}(B_{\wP, }\nu _{\wP})\to
H_{4p-3}(B_{\wP};\bZ  )$.
\end{theorem}
\begin{proof}
Lemma \ref{LemmaForLessThan4p-4}, Lemma \ref{LemmaFor4p-4} and Lemma \ref{LemmaFor4p-3} shows that all the
differentials going out of $E_{4p-3,0}^{r}(\nu _{\wP})$\ in
the James spectral sequence for $\nu _{\wP}$\ are zero on  $T_{\wP}$  and
the result follows.
\end{proof}

\section{Surgery on the bordism element}
In this section we fix an odd prime $p$, the integer $n=2p-1$, and assume that $G$ is a finite subgroup of $\wP $ that maps surjectively onto the quotient $Q_p$ of $\wP $ by $S^1$. We have now completed the first two steps in the proof of Theorem A. We have shown that there is a non-empty subset $T_{\wP}$ consisting of primitive elements of infinite order in $H_{2n-1}(B_{\wP})$, and that this subset is contained in the the image of the Hurewicz map 
$\Omega _{2n-1}(B_{\wP },\nu _{\wP})\to H_{2n-1}(B_{\wP};\bZ  ) $.
We now define the subset
$$T_{G} = \{ \trf(\gamma) \in H_{2n}(B_{G};\bZ) \vv\gamma \in T_{\wP} \}, $$ 
where $\trf\colon H_{2n-1}(B_{\wP}; \bZ) \to H_{2n}(B_{G}; \bZ)$ denotes the $S^1$-bundle transfer induced by the fibration $S^1 \to B_{G} \to B_{\wP}$. Now fix $$\gamma_G \in T_G\ . $$
By definition $\gamma_G = \trf(\gamma)$, for some $\gamma \in T_{\wP}$, so we can pull back the $S^1$-bundle over a manifold (provided by Theorem \ref{hurewicz}) whose fundamental class represents $\gamma$ under the bordism Hurewicz map. Hence we have a bordism class 
$$[M^{2n}, f]\in \Omega_{2n}(B_G, \nu_G) \text{ \ such that \ \ } \gamma_G =f_*[M]\ .$$ 
Surgery will be used to improve the manifold $M$ within its bordism class.
Our first remark is that we may assume $f$ is an $n$-equivalence
(see \cite[Cor.~1, p.~719]{kreck3}). In particular, $\pi_1(M) = G$, and $\pi_i(M) =0$  for $2\leq i < n$. In addition, the map $f_*\colon \pi_n(M) \to \pi_n(B_G)$ is surjective. We need to determine the structure of $\pi_n(M)$ as a $\ZG$-module. First by applying the construction of \cite[p.~230]{benson-carlson1} to the chain complex
$C(\widetilde{B_G})$ we get two $\bZ G$-chain complexes $C(\theta_1)$ and $C(\theta_2)$ as in \cite{benson-carlson1} (see Cor. 4.5 and Remark 3, p. 231), and investigated further in \cite{benson-carlson2}, with the following properties:
\begin{enumerate}
\addtolength{\itemsep}{0.1\baselineskip}
\renewcommand{\labelenumi}{(\roman{enumi})}
\item We have $\theta_1=\Res^{\wP}_{G}(\zeta) \text{ and } \theta_2=\Res^{\wP}_{G}(\alpha ^{p}-\alpha ^{p-1}\beta +\beta ^{p})$.
\item There is a $\bZ G$-chain map $$\psi_i\colon C_*(\widetilde{B_G})\longrightarrow C(\theta_i), \quad \text{for\ } i = 1,2.$$ 
\item  $H_*(C(\theta_i); \bZ) = H_*(S^n ;\bZ)\text{ and }H^*(C(\theta_i); \bZ) = H^*(S^n ;\bZ)$, for $i=1, 2$.
\item There exists $[C(\theta_i)]$ a generator of $H^n(C(\theta_i); \bZ)=\bZ$ such that 
$$(\psi_i)^*([C(\theta_i)])=z_i, \quad i=1,2$$
where $H^{n}(\widetilde{B_{G}};\bZ)\cong H^{n}(K;\bZ)=\la z_1,z_2\ra\cong\bZ \oplus \bZ $.
\item All the modules in the chain complex $$D_* = C(\theta_1)\otimes_{\bZ} C(\theta_2)$$ are finitely-generated projective $\bZ G$-modules.
\end{enumerate}
We will compare this complex to the complex $C_*(S(\Psi_G)) \otimes_\bZ C(\theta_2)$.
\begin{lemma} The modules $C_i(S(\Psi_G)) \otimes_\bZ C_j(\theta_2)$ are finitely-generated, projective $\ZG$-modules.
\end{lemma}
\begin{proof} The module $C_j(\theta_2)$ is free for $j<n$ and  $C_i(S(\Psi_G))$ is free for $i>2$. For $i\leq 2$, $C_i(S(\Psi_G))$ is a direct sum of free modules and modules of the form $\bZ[G/D_t]$ for some $t$. Hence it is enough to show that 
$\bZ[G/D_t]\otimes_\bZ C_n(\theta_2)$ is a projective module for each $t$. We will use the criterion of \cite[VI,8.10]{brown1}: projectivity follows from cohomological triviality. There is an exact sequence $$0 \to L_{\theta_2} \to \Omega^{n+1}\bZ \xrightarrow{\theta_2} \bZ \to 0$$
and $$0 \to L_{\theta_2} \to F_n \to C_n(\theta_2) \to 0$$ where $F_n$ is a free $\ZG$-module. Therefore, it is enough to show that the cohomology groups
$$\widehat H^{q}(G; \bZ[G/D_t]\otimes_\bZ L_{\theta_2}) = 0, \quad \text{for\ large\ } q \in \bZ\ .$$
But we have an isomorphism (using the complete $\Ext$-theory)
$$\wExt^q_{\ZG}(\bZ[G/D_t], L_{\theta_2}) \cong \widehat H^{q}(G; \bZ[G/D_t]\otimes_\bZ L_{\theta_2})$$
by \cite[III,2.2]{brown1}. Now the long exact sequence
$$\xymatrix{\wExt^q_{\ZG}(\bZ[G/D_t], L_{\theta_2}) \ar[r]& 
\wExt^q_{\ZG}(\bZ[G/D_t], \Omega^{n+1}\bZ)\ar[r]\ar@{=}[d] & 
\wExt^q_{\ZG}(\bZ[G/D_t], \bZ)\cr
&\wExt^{q-n-1}_{\ZG}(\bZ[G/D_t], \bZ)\ar[ur]_{\cup\, \theta_2} &
}$$
combined with Shapiro's Lemma \cite[III,6.2]{brown1}
$$\wExt^q_{\ZG}(\bZ[G/D_t], \bZ) = \wExt^{q}_{\bZ D_t}(\bZ, \bZ),$$
 and the fact that $\Res^G_{D_t}( \theta_2) \in H^{n+1}(D_t; \bZ)$ is a generator, completes the proof.
\end{proof}
\begin{lemma}\label{lem: D is free}
$D_*$ is chain homotopy equivalent to a finite free $\bZ G$-chain complex
\end{lemma}
\begin{proof} As in \cite{benson-carlson1} we have the following pushout diagram

$$\xymatrix{         & L_{\theta_1}\ar[d]\ar@{=}[r] & L_{\theta_1}\ar[d]  \cr 
0\ar[r]  & \Omega^{n+1}\bZ\ar[d]^{\theta_1}\ar[r] & C_n(\widetilde{BG})\ar[d]\ar[r] & C_{n-1}(\widetilde{BG})\ar[d]\ar[r] & \dots 
\ar[r] & C_0(\widetilde{BG})\ar[d]\ar[r] & \bZ \ar[d]^{id}\ar[r] & 0 \cr
0\ar[r]  & \bZ \ar[r] & C_n(\widetilde{BG})/L_{\theta_1}\ar[r] & C_{n-1}(\widetilde{BG})\ar[r] & \dots \ar[r] & C_0(\widetilde{BG})\ar[r]& \bZ \ar[r] & 0 
}$$
where the lower row is the chain complex $C(\theta_1)$.

If one extends the identity on $\bZ$ 's on the right hand side of the diagram below to a chain map $C_*(\widetilde{BG})\to C_*(S(\Psi_G))$ then the map 
on the left hand side $\Omega^{n+1}\bZ\to \bZ$ must also represent $\theta_1 $ in $H^{n+1}(BG)$. 

$$\xymatrix@C-2pt{ 0\ar[r]  & \Omega^{n+1}\bZ\ar[d]^{\theta_1}\ar[r] & C_n(\widetilde{BG})\ar[d]\ar[r] & C_{n-1}(\widetilde{BG})\ar[d]\ar[r] & \dots 
\ar[r] & C_0(\widetilde{BG})\ar[d]\ar[r]^(0.6){\varepsilon} & \bZ \ar[d]^{id}\ar[r] & 0 \cr
0\ar[r]  & \bZ \ar[r] & C_n(S(\Psi_G))\ar[r] & C_{n-1}(S(\Psi_G))\ar[r] & \dots \ar[r] & C_0(S(\Psi_G))\ar[r]^(0.6){\varepsilon}& \bZ \ar[r] & 0 }$$

This is because the class $\zeta \in H^{2p}(B\wP;\bZ)$  is the unique cohomology class $u$ in this dimension such that $\Res^\wP_{D_t}(u) = 0$ for
$0 \leq t \leq p$, and $\Res^\wP_{S^1}(u) = \tau^p$. On the other hand, by construction each subgroup $D_t \cong \cy p$ is an isotropy subgroup of the action on $S(\Psi_G)$. The fixed-point complex $C_*(S(\Psi_G)^{D_t})$ has the homology of an odd-dimensional sphere (of lower dimension). Therefore, after restriction to $D_t$ we can lift the identity on $\bZ$ using $$C_0(\widetilde{BG})= \ZG \xrightarrow{\varepsilon} \bZ \subseteq C_0(S(\Psi_G)^{D_t})\subset C_0(S(\Psi_G))$$
and this lifting extends to the zero map $\Omega^{n+1}\bZ \to \bZ$.

Notice that these diagrams provide the translation between equivalence classes of multiple extensions and cohomology classes, as described in \cite[III, 6.4]{maclane1}. Since $C(\theta_1)$ and $C_*(S(\Psi_G))$ considered as $n$-fold extensions from $\bZ$ to $\bZ$ both represent the same cohomology class, there is a chain map 
$$C(\theta_1)\to C_*(S(\Psi_G))$$
which induces a cohomology isomorphism. 
Hence by the Kunneth formula we have a cohomology isomorphism
$$C(\theta_1)\otimes_\bZ C(\theta_2)\to C_*(S(\Psi_G))\otimes_\bZ C(\theta_2)$$ 
where all the modules in $C_*(S(\Psi_G))\otimes C(\theta_2)$ are projective. Therefore we have a chain homotopy equivalence of finitely-generated projective $\bZ G$-chain complexes 
$$D_*\to C_*(S(\Psi_G))\otimes_\bZ C(\theta_2)\ .$$ 
However, in the chain complex $C_*(S(\Psi_G))\otimes C(\theta_2)$, all possibly non-free-modules projective modules have the form $\Ind_{D_t}^G(C_i(S(\Psi_G)^{D_t})) \otimes_\bZ C_n(\theta_2)$. Since the Euler characteristic  $\chi(C_*(S(\Psi_G)^{D_t}))=0$, the finiteness obstruction of $C_*(S(\Psi_G))\otimes_\bZ C(\theta_2)$ vanishes.
\end{proof}
\begin{lemma}\label{hyperplane}
Under the transfer $\tr\colon H_{2n}(B_{G};\bZ)
\to H_{2n}(\widetilde{B_{G}};\bZ)$,  the class $\tr(\gamma_G )$ corresponds to the standard hyperbolic form 
$$\bH(\bZ)=( \bZ\oplus \bZ, \mmatrix{{\hphantom{-}0}}{1}{-1}{0})$$
on $\pi_n(B_{G}) = \bZ\oplus \bZ$ under the identification
$H_{2n}(\widetilde{B_{G}};\bZ)/Tors \cong \Gamma(\bZ\oplus \bZ)$
with Whitehead's $\Gamma$-functor.
\end{lemma}
\begin{proof}
Let $d = |G|/p^2$ denote the order of the centre of $G$. We have a commutative diagram
$$\xymatrix{
S^1\ar[r]^d\ar[d]&S^1\ar[r]\ar[d]&B\cy{d}\ar[d]\\
\widetilde{B_G}\ar[d]\ar[r]&B_G\ar[r]\ar[d]&BG\ar[d]\\
B_{S^1}\ar[r]&B_{\wP}\ar[r]&BQ}$$
where $Q = \cy p \times \cy p$ and $\widetilde{B_G}= K(\bZ \oplus \bZ, n)$. This gives a commutative square
$$\xymatrix{H_{2n-1}(B_{S^1}; \bZ) \ar[r]^{\trf}&H_{2n}(\widetilde{B_G}; \bZ)\\
H_{2n-1}(B_{\wP}; \bZ )\ar[u]^{tr}\ar[r]^{\trf}&H_{2n}(B_G; \bZ)\ar[u]^{tr}
}$$
relating the $S^1$-bundle transfers and the universal covering transfers. 
The $p$-torsion subgroup of $H_{2n-1}(B_{S^1};\bZ)$ maps to zero under the  $S^1$-bundle transfer, since $H_{2n}(K;\bZ)$ has no $p$-torsion, by Lemma 
\ref{CohomologyOfK}.
Therefore $tr(\gamma_G)= \trf (\gamma_{S^1})$ is just the image of the 
fundamental class of $S^{2p-1}\times S^{2p-1}$ in 
$H_{2n}(\widetilde{B_G}; \bZ) = H_{2n}(K; \bZ )$. But $H_{2n}(K;\bZ)/Tors = \bZ$ can be naturally identified with $\Gamma(\bZ\oplus \bZ) = \bZ$, and under this identification the fundamental class of $S^{2p-1}\times S^{2p-1}$ corresponds to a generator, represented by the hyperbolic plane.
\end{proof}
\begin{theorem}
The equivariant intersection form of $M$ is in the following form
$$(\pi_n(M), s_M) \cong 
\bH(\bZ) \perp (F, \lambda)$$
where $(F,\lambda)$ is a non-singular skew-hermitian form 
on a finitely-generated free $\bZ G$-module. 
\end{theorem}
\begin{proof}
Let $\psi\colon C_*(\widetilde{M}) \to D_*$ be the following composition
$$  C_*(\widetilde{M}) \to C_*(\widetilde{B_{G}}) \xrightarrow{\ \Delta\  } C_*(\widetilde{B_{G}}) \otimes C_*(\widetilde{B_{G}}) \xrightarrow{\ \psi_1   \otimes \psi_2\  } D_*$$
where $\Delta $ denotes the diagonal map. First note that $\psi_*\colon H_i(\widetilde{M}) \to H_i(D_*)$ is clearly surjective for $i<2n$. Assume 
$\psi_*([\widetilde{M}])=[D]$ then 
$$1=\la z_1\cup z_2 , f_*([\widetilde{M}]) \ra = \la [C(\theta_1)] \otimes [C(\theta_2)],[D] \ra $$
by the Kunneth formula. Hence $\psi_*$ is also is surjective for $i \geq 2n$. As the image of the fundamental class $[M]$ of maps to a generator of $H_{2n}(D_*)$.
Hence the homology of the mapping cone $H_i(\psi )$ is zero for $i\neq n$, and  $H_n(\psi )=P$ is a finitely generated projective $\bZ G$-module. But $P$ is stably free by 
Lemma \ref{lem: D is free}. Hence $\pi_n(M)=\bZ \oplus\bZ \oplus P $ where $P$ is a finitely generated . By stabilizing $M$ with connected sums of copies of $S^{2p-1}\times S^{2p-1}$ we may assume that $\pi_n(M) = \bZ \oplus\bZ \oplus F$, where $F$ is a finitely-generated free $\bZ G$-module.

To show the splitting of the equivariant intersection form $(\pi_n(M), s_M)$ we consider the relation
 $$\langle f^*(z_1)\cup f^*(z_2), [\widetilde M]\rangle = 
 \langle z_1\cup z_2, f_*[\widetilde M]\rangle$$
 where $z_1$, $z_2$ are a symplectic basis for the form on $\pi_n(B_G)$. Therefore, by Lemma \ref{hyperplane}, the map
 $f^*\colon H^n(\widetilde{B_G};\bZ) \to H^n(\widetilde M;\bZ)$ gives an isometric embedding of the hyperbolic form $\bH(\bZ)$ into $s_M$. Any such isometric embedding splits (see \cite[Lemma 1.4]{hnk1}).
Hence the result follows.
\end{proof}
We next observe that the equivariant intersection form $(\pi_n(M), s_M)$ has a quadratic refinement $\mu\colon \pi_n(M) \to \ZG/\{\nu + \bar\nu\}$, in the sense of \cite[Theorem 5.2]{wallbook}. This follows because the universal covering $\widetilde M$ has stably trivial normal bundle (use the Browder-Livesay quadratic map \cite[Lemma 4.5, 4.6]{browder-livesay1} for the elements of order two in $G$). We therefore obtain an element
$(F, \lambda, \mu)$
of the surgery obstruction group (see \cite[p.~49]{wallbook} for the essential definitions). The  Arf invariant of this form is the Arf invariant of the associated form $\epsilon_*(F, \lambda, \mu)$, where $\epsilon\colon \ZG \to \bZ$ is the augmentation map. This invariant factors through 
$$\Omega_{2n}(B_G, \nu_G) \to \Omega_{2n}(B_\wP, \nu_\wP) \xrightarrow{Arf} \cy 2$$
and hence is zero for our bordism element. 
We also need to check the discriminant of this form.
\begin{lemma}
We obtain an element
$$(F, \lambda, \mu) \in L'_{2n}(\ZG)$$
of the weakly-simple surgery obstruction group.
\end{lemma}
\begin{proof}
A non-singular, skew-hermitian quadratic form $(F, \lambda, \mu)$ represents an element in 
$L'_{2n}(\ZG)$ provided that its discriminant lies in $\ker(\wh(\ZG) \to \wh(\QG))$. But the equivariant symmetric Poincar\'e chain complex $(C(M), \varphi_0)$ is chain equivalent, after tensoring with the rationals $\bQ$, to the rational homology (see \cite[\S 4]{ra10}). Therefore the image of the discriminant of $(\pi_n(M)\otimes \bQ, s_M)$ equals the image of the torsion of $\varphi_0$, which vanishes in $\wh(\QG)$ because closed manifolds have simple Poincar\'e duality (see \cite[Theorem 2.1]{wallbook}).
\end{proof}
\begin{proof}[The proof of Theorem A] Suppose that $p$ is an odd prime.
We now have a representative $[M,f]$ for our bordism element in $\Omega_{2n}(B_G,\nu_G)$ whose equivariant intersection form $(\pi_n(M), s_M)$ contains $(F, \lambda, \mu)$ as described above. However, an element in the surgery obstruction group ${L'}_{2n}(\ZG)$
is zero provided that its multisignature and ordinary Arf invariant both vanish
 This is a result of Bak and Wall for groups of odd order (see \cite[Cor.~2.4.3]{wall-VI}), and for odd order groups direct product with cyclic groups we apply
 \cite[Theorem 2.4.2 and Cor.~3.3.3]{wall-VI}. The multisignature invariant is trivial since $M$ is a closed manifold \cite[13B]{wallbook}. The ordinary Arf invariant of the universal covering $\widetilde M$ vanishes since $2n=4p-2$ is not of the form $2^k -2$ (a famous result of Browder \cite{browder1}). We can now do surgery on the classifying map $f\colon M \to B_G$  respecting the bordism class in $\Omega_{2n}(B_G,\nu_G)$  to obtain a representative $[M,f]$ which has $\widetilde M = S^n\times S^n \# \Sigma$, where $\Sigma$ is a homotopy $2n$-sphere. Since the $p$-primary component of Cok$\,J$
 starts in dimension $2p(p-1)-2$ (see  \cite[p.~5]{ravenel1}) we can eliminate this homotopy sphere by equivariant connected sum unless $p=3$.
 
In case $p=3$, we use Remark \ref{homotopysphere} to show that $\widetilde M = S^5\times S^5$. The bordism element $[\widetilde M, \widetilde f]\in \Omega^{fr}_{10}(K)$ vanishes in $\Omega_{10}(B_{S^1}, \nu_{S^1})$ by the Gysin sequence in bordism. But the difference element $[\widetilde M, \widetilde f] - [S^5\times S^5, i_{5}] \in \Omega^{fr}_{10}(\ast)$. Since $\Omega^{fr}_{10}(\ast)$ injects on the $3$-component into $\Omega_{10}(BS^1, \xi_{S^1})$, it follows that the order of the difference element is not divisible by $3$. Thus in all cases we can obtain $\widetilde M = S^n \times S^n$. This completes the proof of Theorem A.
\end{proof}


%

\begin{thebibliography}{10}

\bibitem{adem1}
A.~Adem, \emph{Constructing and deconstructing group actions}, Homotopy theory:
  relations with algebraic geometry, group cohomology, and algebraic
  $K$-theory, Contemp. Math., vol. 346, Amer. Math. Soc., Providence, RI, 2004,
  pp.~1--8.

\bibitem{adem-davis-unlu}
A.~Adem, J.~F. Davis, and {\"O}.~{\"U}nl{\"u}, \emph{Fixity and free group
  actions on products of spheres}, Comment. Math. Helv. \textbf{79} (2004),
  758--778.

\bibitem{adem-smith}
A.~Adem and J.~H. Smith, \emph{Periodic complexes and group actions}, Ann. of
  Math. (2) \textbf{154} (2001), 407--435.

\bibitem{alzubaidy1}
K.~Alzubaidy, \emph{Free actions of {$p$}-groups {$(p>3)$} on {$S\sp{n}\times
  S\sp{n}$}}, Glasgow Math. J. \textbf{23} (1982), 97--101.

\bibitem{alzubaidy2}
\bysame, \emph{Free actions on {$(S\sp n)\sp k$}}, Mathematika \textbf{32}
  (1985), 49--54.

\bibitem{benson-carlson2}
D.~J. Benson and J.~F. Carlson, \emph{Projective resolutions and {P}oincar\'e
  duality complexes}, Trans. Amer. Math. Soc. \textbf{342} (1994), 447--488.

\bibitem{benson-carlson1}
D.~J. Benson and J.~F. Carlson, \emph{Complexity and multiple complexes}, Math.
  Z. \textbf{195} (1987), 221--238.

\bibitem{blackburn2}
N.~Blackburn, \emph{Generalizations of certain elementary theorems on
  {$p$}-groups}, Proc. London Math. Soc. (3) \textbf{11} (1961), 1--22.

\bibitem{browder-livesay1}
W.~Browder and G.~R. Livesay, \emph{Fixed point free involutions on homotopy
  spheres}, T\^ohoku Math. J. (2) \textbf{25} (1973), 69--87.

\bibitem{browder1}
W.~Browder, \emph{The {K}ervaire invariant of framed manifolds and its
  generalization}, Ann. of Math. (2) \textbf{90} (1969), 157--186.

\bibitem{brown1}
K.~S. Brown, \emph{Cohomology of groups}, Springer-Verlag, New York, 1994,
  Corrected reprint of the 1982 original.

\bibitem{cartan2}
H.~Cartan, \emph{Sur les groupes d'{E}ilenberg-{M}ac {L}ane. {II}}, Proc. Nat.
  Acad. Sci. U. S. A. \textbf{40} (1954), 704--707.

\bibitem{cartan3}
\bysame, \emph{S\'eminaire {H}enri {C}artan de l'{E}cole {N}ormale
  {S}up\'erieure, 1954/1955. {A}lg\`ebres d'{E}ilenberg-{M}ac{L}ane et
  homotopie}, Secr\'etariat math\'ematique, 11 rue Pierre Curie, Paris, 1955.

\bibitem{conner1}
P.~E. Conner, \emph{On the action of a finite group on {$S\sp{n}\times
  S\sp{n}$}}, Ann. of Math. (2) \textbf{66} (1957), 586--588.

\bibitem{glover1}
D.~J. Glover, \emph{A study of certain modular representations}, J. Algebra
  \textbf{51} (1978), 425--475.

\bibitem{h4}
I.~Hambleton, \emph{Some examples of free actions on products of spheres},
  Topology \textbf{45} (2006), 735--749.

\bibitem{h-unlu1}
I.~Hambleton and O.~\"Unl\"u, \emph{Examples of free actions on products of
  spheres}, Quart. J. Math.  (to appear), arXiv:math.AT 0705.4081, 2008.

\bibitem{heller1}
A.~Heller, \emph{A note on spaces with operators}, Illinois J. Math. \textbf{3}
  (1959), 98--100.

\bibitem{hnk1}
F.~Hirzebruch, W.~D. Neumann, and S.~S. Koh, \emph{Differentiable manifolds and
  quadratic forms}, Marcel Dekker Inc., New York, 1971, Appendix II by W.
  Scharlau, Lecture Notes in Pure and Applied Mathematics, Vol. 4.

\bibitem{jackson1}
M.~A. Jackson, \emph{{${\rm Qd}(p)$}-free rank two finite groups act freely on
  a homotopy product of two spheres}, J. Pure Appl. Algebra \textbf{208}
  (2007), 821--831.

\bibitem{kambe1}
T.~Kambe, \emph{The structure of {$K\sb{\Lambda }$}-rings of the lens space and
  their applications}, J. Math. Soc. Japan \textbf{18} (1966), 135--146.

\bibitem{karoubi1}
M.~Karoubi, \emph{{$K$}-theory}, Springer-Verlag, Berlin, 1978, An
  introduction, Grundlehren der Mathematischen Wissenschaften, Band 226.

\bibitem{kreck3}
M.~Kreck, \emph{Surgery and duality}, Ann. of Math. (2) \textbf{149} (1999),
  707--754.

\bibitem{leary1}
I.~J. Leary, \emph{The integral cohomology rings of some {$p$}-groups}, Math.
  Proc. Cambridge Philos. Soc. \textbf{110} (1991), 25--32.

\bibitem{leary2}
\bysame, \emph{The mod-{$p$} cohomology rings of some {$p$}-groups}, Math.
  Proc. Cambridge Philos. Soc. \textbf{112} (1992), 63--75.

\bibitem{leary3}
I.~J. Leary, \emph{The cohomology of ceratin groups}, Ph.D. thesis, University
  of Cambridge, 1991,
  (http://www.maths.abdn.ac.uk/~bensondj/html/archive/leary.html).

\bibitem{lewis2}
G.~Lewis, \emph{Free actions on {$S\sp{n}\times S\sp{n}$}}, Trans. Amer. Math.
  Soc. \textbf{132} (1968), 531--540.

\bibitem{lewis1}
\bysame, \emph{The integral cohomology rings of groups of order {$p\sp{3}$}},
  Trans. Amer. Math. Soc. \textbf{132} (1968), 501--529.

\bibitem{maclane1}
S.~Mac~Lane, \emph{Homology}, Classics in Mathematics, Springer-Verlag, Berlin,
  1995, Reprint of the 1975 edition.

\bibitem{milnor-stasheff}
J.~W. Milnor and J.~D. Stasheff, \emph{Characteristic classes}, Princeton
  University Press, Princeton, N. J., 1974, Annals of Mathematics Studies, No.
  76.

\bibitem{oliver1}
R.~Oliver, \emph{Free compact group actions on products of spheres}, Algebraic
  topology, Aarhus 1978 (Proc. Sympos., Univ. Aarhus, Aarhus, 1978), Lecture
  Notes in Math., vol. 763, Springer, Berlin, 1979, pp.~539--548.

\bibitem{ra10}
A.~Ranicki, \emph{The algebraic theory of surgery. {I}. {F}oundations}, Proc.
  London Math. Soc. (3) \textbf{40} (1980), 87--192.

\bibitem{ravenel1}
D.~C. Ravenel, \emph{Complex cobordism and stable homotopy groups of spheres},
  AMS Chelsea Publishing, 2004.

\bibitem{stein1}
E.~Stein, \emph{Free actions on products of spheres}, Michigan Math. J.
  \textbf{26} (1979), 187--193.

\bibitem{stong1}
R.~E. Stong, \emph{Notes on cobordism theory}, Mathematical notes, Princeton
  University Press, Princeton, N.J., 1968.

\bibitem{teichner1}
P.~Teichner, \emph{On the signature of four-manifolds with universal covering
  spin}, Math. Ann. \textbf{295} (1993), 745--759.

\bibitem{thomas1}
C.~B. Thomas, \emph{Free actions by {$p$}-groups on products of spheres and
  {Y}agita's invariant {$po(G)$}}, Transformation groups (Osaka, 1987), Lecture
  Notes in Math., vol. 1375, Springer, Berlin, 1989, pp.~326--338.

\bibitem{unlu1}
{\"O}.~{\"U}nl{\"u}, \emph{Constructions of free group actions on products of
  spheres}, Ph\~.D~. Thesis, University of Wisconsin, 2004.

\bibitem{wall-VI}
C.~T.~C. Wall, \emph{Classification of {H}ermitian {F}orms. {V}{I}. {G}roup
  rings}, Ann. of Math. (2) \textbf{103} (1976), 1--80.

\bibitem{wallbook}
\bysame, \emph{Surgery on compact manifolds}, second ed., American Mathematical
  Society, Providence, RI, 1999, Edited and with a foreword by A. A. Ranicki.

\bibitem{yagita1}
N.~Yagita, \emph{On the dimension of spheres whose product admits a free action
  by a nonabelian group}, Quart. J. Math. Oxford Ser. (2) \textbf{36} (1985),
  117--127.

\end{thebibliography}
\providecommand{\bysame}{\leavevmode\hbox to3em{\hrulefill}\thinspace}
\providecommand{\MR}{\relax\ifhmode\unskip\space\fi MR }
\providecommand{\MRhref}[2]{%
  \href{http://www.ams.org/mathscinet-getitem?mr=#1}{#2}
}
\providecommand{\href}[2]{#2}

\end{document}